\documentclass[british, a4paper,12pt]{article}
\usepackage[font={small,sf}, labelfont={sf,bf}, width=\textwidth]{caption}

\usepackage[utf8]{inputenc}
\usepackage[british]{babel}

\usepackage{amsthm,amssymb,bbm,stmaryrd}

\usepackage{libertinus}
\usepackage[T1]{fontenc}
\usepackage{libertinust1math}
\usepackage[cal=euler, scr=rsfso]{mathalpha}
\useosf

\usepackage{microtype}

\usepackage[a4paper,vmargin={2cm,2cm},hmargin={2.25cm,2.25cm}]{geometry}

\usepackage{hyperref}
\hypersetup{
	colorlinks=false,
	breaklinks=true,
	pdfpagemode=UseNone,
	bookmarksopen=false,
}
\usepackage{cleveref}

\usepackage{enumitem}

\usepackage{xcolor}
\usepackage{todonotes}

\numberwithin{equation}{section}

\newtheorem{thm}{\bfseries \upshape Theorem}[section]
\newtheorem{lem}[thm]{Lemma}
\newtheorem{prop}[thm]{Proposition}
\newtheorem{cor}[thm]{Corollary}

\theoremstyle{definition}

\newtheorem{defi}[thm]{Definition}

\renewcommand{\P}{\mathop{}\!\mathbb{P}}
\newcommand{\E}{\mathop{}\!\mathbb{E}}
\newcommand{\R}{\mathbb{R}}
\newcommand{\Z}{\mathbb{Z}}

\newcommand{\D}{\mathbb{D}}
\newcommand{\resamp}{\xi}
\newcommand{\Resamp}{\Xi}

\newcommand{\e}{\operatorname{e}} 
\renewcommand{\d}{\mathop{}\!\mathrm{d}}

\newcommand{\cv}[1][n]{\enskip\mathop{\longrightarrow}^{}_{#1 \to \infty}\enskip}
\newcommand{\cvloi}[1][n]{\enskip\mathop{\longrightarrow}^{(d)}_{#1 \to \infty}\enskip}
\newcommand{\cvps}[1][n]{\enskip\mathop{\longrightarrow}^{a.s.}_{#1 \to \infty}\enskip}
\newcommand{\cvproba}[1][n]{\enskip\mathop{\longrightarrow}^{\P}_{#1 \to \infty}\enskip}
\newcommand{\eqloi}[1][n]{\enskip\mathop{=}^{(d)}_{}\enskip}

\usepackage{mathtools}

\DeclareSymbolFont{extraup}{U}{zavm}{m}{n}
\DeclareMathSymbol{\vardspade}{\mathalpha}{extraup}{81}
\DeclareMathSymbol{\varheart}{\mathalpha}{extraup}{86}
\DeclareMathSymbol{\vardiamond}{\mathalpha}{extraup}{87}
\DeclareMathSymbol{\varclub}{\mathalpha}{extraup}{84}

\makeatletter
\renewcommand*{\@fnsymbol}[1]{\ensuremath{\ifcase#1\or \vardspade \or \varheart \or \vardiamond\or \varclub \or
  \mathsection\or \mathparagraph\or \|\or **\or \dagger\dagger
  \or \ddagger\ddagger \else\@ctrerr\fi}}
\makeatother

\author{
Nicolas \textsc{Curien}\thanks{Institut de Math\'ematique d'Orsay,
Universit\'e Paris-Saclay, 91400 Orsay, France.\protect\\
\href{mailto:nicolas.curien@gmail.com}
{\texttt{nicolas.curien@universite-paris-saclay.fr}}}
\qquad\&\qquad
Cyril \textsc{Marzouk}\thanks{CMAP, CNRS, \'Ecole polytechnique,
Institut Polytechnique de Paris, Palaiseau, France.\protect\\
\href{mailto:cyril.marzouk@polytechnique.edu}
{\texttt{cyril.marzouk@polytechnique.edu}}}
}

\date{}

\title{R\'emy's diffusion on Brownian trees}

\begin{document}

\maketitle

\begin{abstract}
R\'emy's algorithm is a famous recursive construction of uniform random binary trees of growing size by a local grafting operation.
In this work we construct a continuous version, a new local diffusion on the space of real trees of growing Brownian Continuum Random Trees (CRT's).
It appears as the scaling limit of a variant of R\'emy's algorithm due to Bacher, Bodini, and Jacquot.
Once the trees are rescaled to have constant mass, this diffusion uncovers an ergodic dynamics on trees with the Brownian CRT as unique invariant law.
\end{abstract}

\begin{figure}[h!]
    \centering
    \vspace{-2cm}
    \includegraphics[width=7cm, trim=0cm 0cm 0cm 8cm,clip]{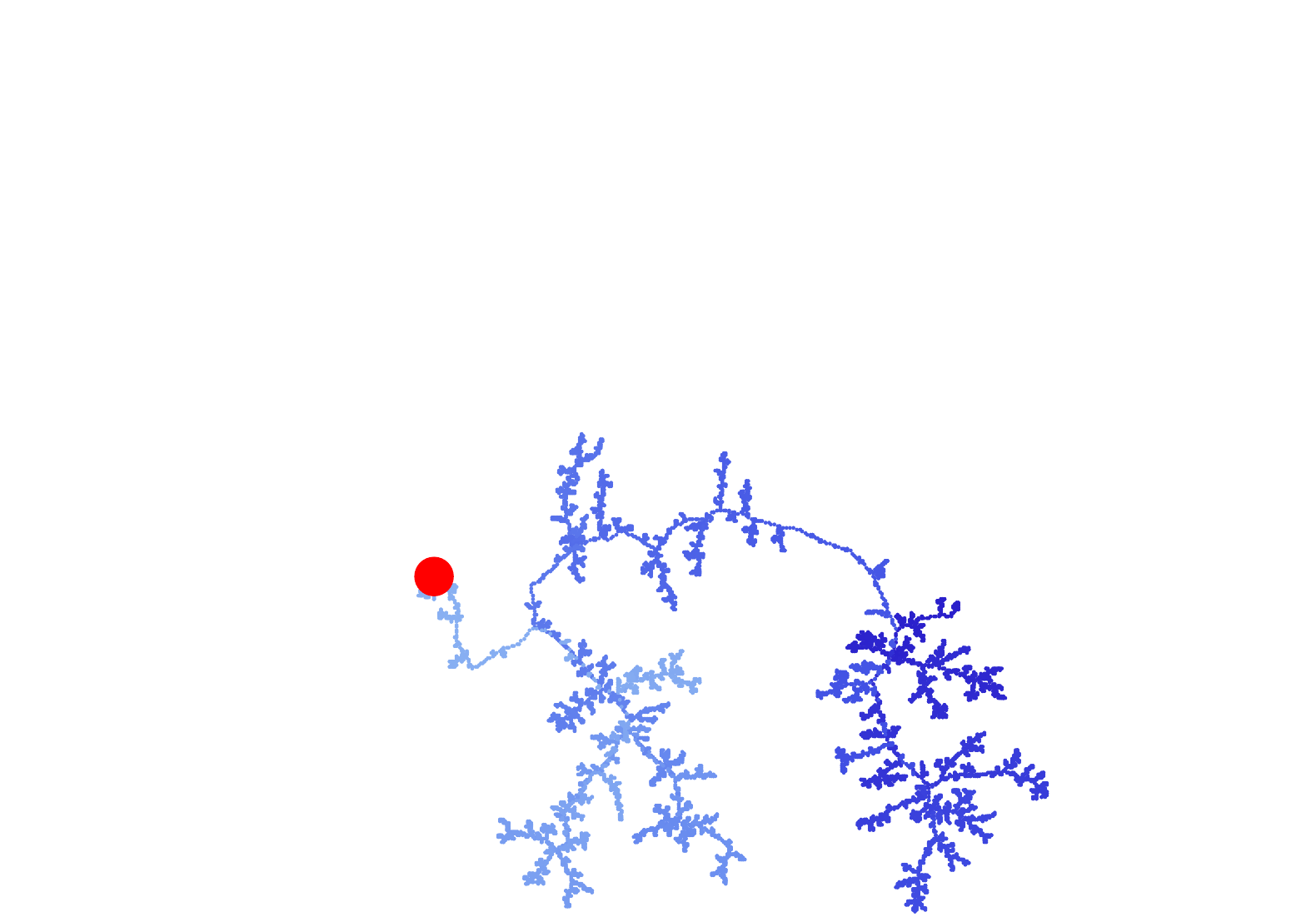}        \includegraphics[width=7cm, trim=0cm 0cm 0cm 8cm,clip]{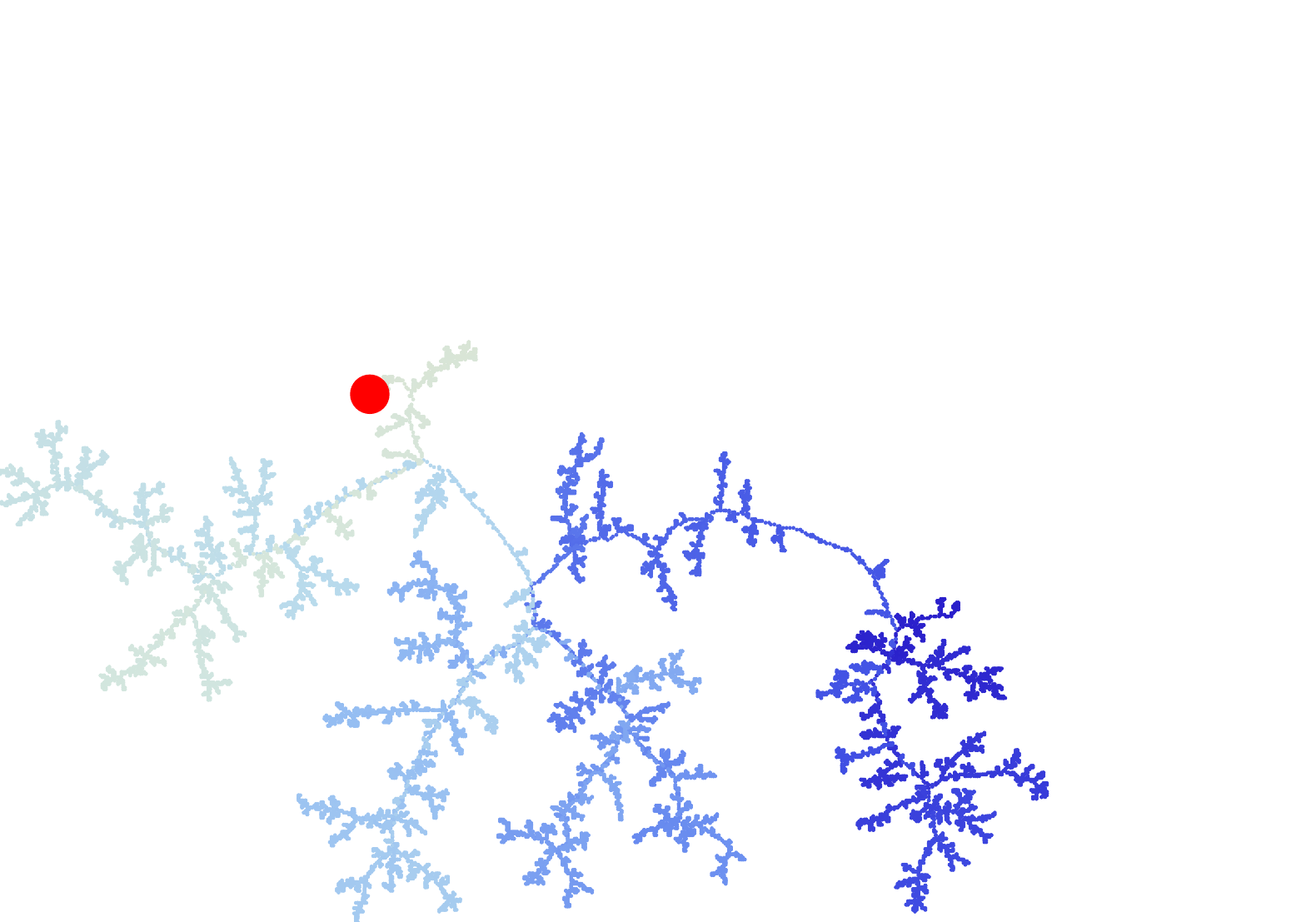}  
    
    \bigskip 
    
           \includegraphics[width=7cm, trim=0cm 0cm 0cm 8cm,clip]{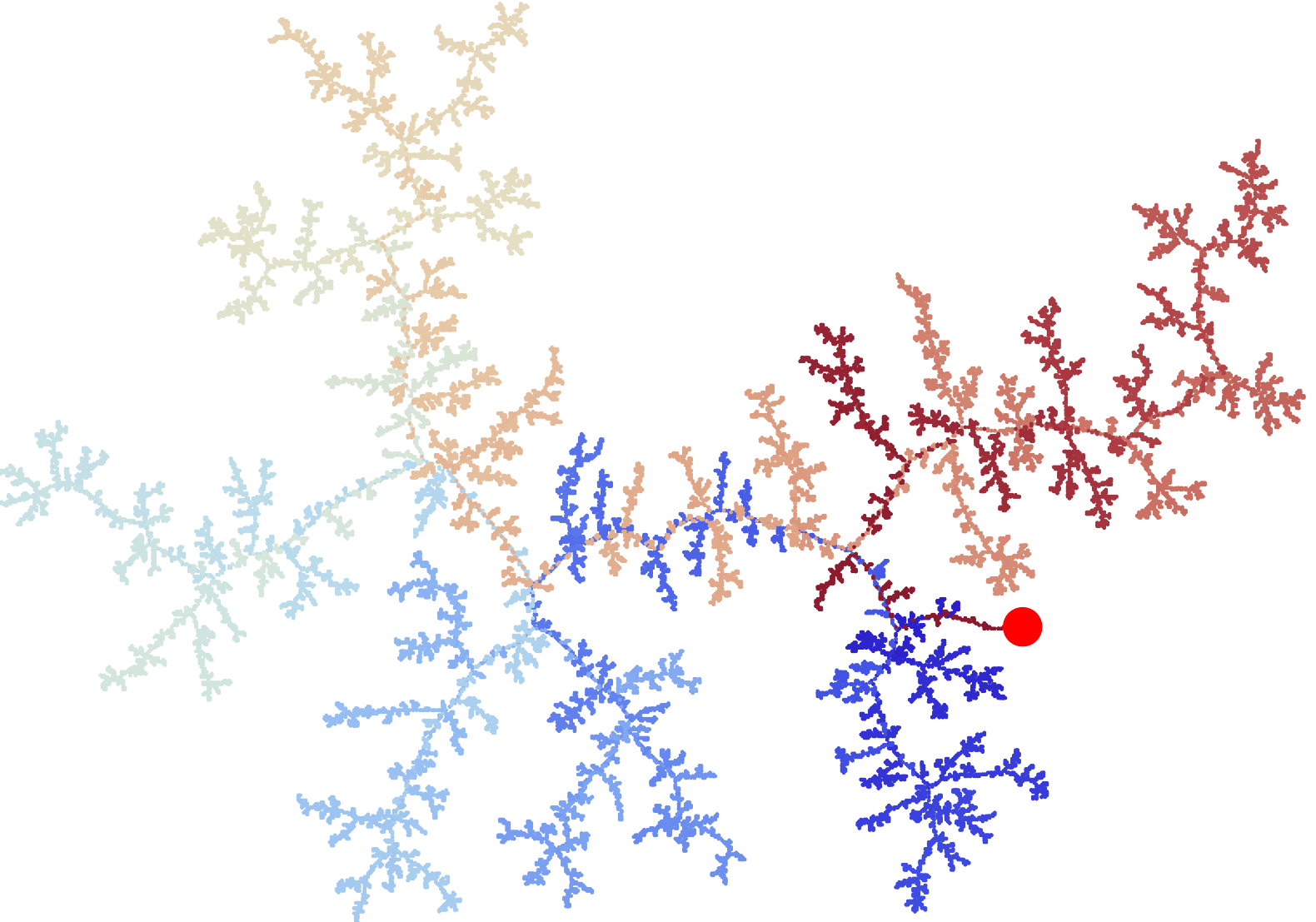}

    \caption{
    Snapshot at times 10k, 20k, and 40k of the R\'emy diffusion. Colours are ranging from dark blue to red as new vertices are created. The red point is the current marked point where the growth continues. \href{https://www.imo.universite-paris-saclay.fr/~nicolas.curien/videos/remy.mp4}{\emph{Click here to see a video.}}
    }
    \label{fig:simu}
\end{figure}

\clearpage
\tableofcontents

\section{Introduction}

This work studies a variant of the celebrated R\'emy algorithm to construct a Markov chain of marked binary trees with growing size. Its scaling limit gives rise to a novel diffusion on marked continuum trees, which we call \textsc{R\'emy's growing diffusion}, and whose unique stationary distribution is the Brownian CRT of Aldous.
In this introduction, we first describe precisely the discrete dynamics, then, more informally, the continuum one, before stating our main results.

\subsection{One-step R\'emy algorithm \& Bacher--Bodini--Jacquot projection}

Throughout this work, all discrete trees that we consider will be rooted plane binary trees (each vertex has $0$ or $2$ children) and we shall simply call them \textbf{binary trees}. The \textbf{size} $n$ of the tree will be its number of internal vertices (it then has $n+1$ leaves). We shall often distinguish a vertex, in which case the tree, and this vertex, are said to be \textbf{marked}.
R\'emy's algorithm~\cite{Rem85} is a recursive construction of random binary trees based on the following one-step growth:
Given a binary plane tree $T_n$ with size $n$ and a marked vertex $X_n$ that can be either an internal vertex or a leaf, one constructs a random tree $T_{n+1}$ by replacing in $T_n$ the vertex $X_n$ by a new internal vertex with two offspring: one is $X_n$ and the other one is a new leaf, which is placed either to the left or to the right of $X_n$ with probability $1/2$ each, see Figure~\ref{fig:remy} left.

It is straightforward to check that if $T_n$ is uniformly distributed on binary trees with size $n$ and if, conditionally on $T_n$, the vertex $X_n$ is uniformly distributed over the $2n+1$ possible vertices, then $T_{n+1}$ has the uniform distribution on binary trees with size $n+1$.
A famous theorem originally due to Aldous~\cite{Ald93} (see also~\cite{LG05} for this version) states that such a random binary tree $T_n$ with size $n$ satisfies the scaling limit:
\begin{equation}\label{eq:aldous}
\frac{1}{\sqrt{2n}} \cdot T_n \cvloi \boldsymbol{\mathcal{T}}_{2\mathbb e}, \end{equation}
in the Gromov--Hausdorff--Prokhorov sense where the limit is Aldous' original Brownian CRT, coded by \textbf{twice} the standard Brownian excursion $\mathbb e$; this excursion is informally Brownian motion on $[0,1]$ conditioned to stay nonnegative and end at $0$. The goal of this work is to define the analogue of the discrete R\'emy algorithm on Brownian trees obtained heuristically by ``grafting and stretching'' locally in $\mathcal{T}_{2 \mathbb{e}}$ as in the discrete Figure~\ref{fig:remy} left.
We shall do so by taking appropriate scaling limit of the discrete R\'emy algorithm.

\begin{figure}[h!]\centering
\includegraphics[page=1, height=6\baselineskip]{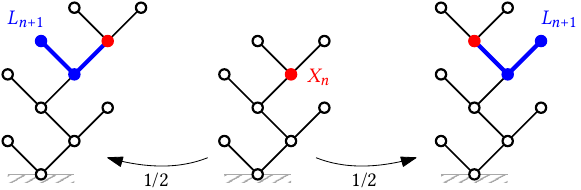}
\qquad
{\vrule width 1pt height 6\baselineskip}
\qquad
\includegraphics[page=2, height=6\baselineskip]{les_dessins}
\caption{The two steps of the BBJ dynamics.
Left: we replace the marked red vertex $X_n$ in $T_n$ by a blue cherry and choose the side of the new leaf $L_{n+1}$ with probability $1/2$.
Right: the bijective projection of a leaf on the internal vertices, except the rightmost leaf whose projection ``exits the tree'', which is used to determine the next marked vertex.
}
\label{fig:remy}
\end{figure}

\subsubsection*{The Bacher--Bodini--Jacquot coupled version} 

In the original \textbf{R\'emy algorithm}, the one-step growth as in Figure~\ref{fig:remy} left is iterated at each step $n$ by choosing independently of the rest the marked vertex $X_n$ uniformly at random in $T_n$ and performing the growth procedure at this vertex to get $T_{n+1}$. 
Unfortunately, the Markov chain $(T_n)_{n \ge 0}$ that is produces converges \emph{almost surely} after suitable normalisation towards a Brownian CRT, see~\cite[Exercises~7.4.10 and~7.4.11]{Pit06} or~\cite[Theorem~5]{CH13}, so this chain \emph{cannot} be used to define a continuous version. 

However, one can also couple the growing points $X_n$ and $X_{n+1}$ while preserving the crucial property that $X_n$ is a uniform random vertex of $T_n$ at each time $n$.
In~\cite{BBJ17}, \textbf{Bacher, Bodini, and Jacquot} introduced such a coupling
in such a way that in some sense $X_{n+1}$ is as close to $X_n$ as possible.
Their work was motivated by the fact that this coupling is nearly optimal in terms of randomness, for it requires $\sim 2n$ bits to generate a random tree of size $n$, which best possible up to lower order corrections, see~\cite{BBJ17}. 
Our motivation is that this coupling  makes the growth steps accumulate locally and this will provide a nontrivial scaling limit.
Let us now describe this coupling, which is represented in Figure~\ref{fig:remy} right.

Since a binary tree of size $n$ has $n+1$ leaves,  it is possible to associate in a bijective way every leaf but one with an internal vertex:
Every vertex of the tree is either a right vertex or a left vertex, the root being a right vertex by convention.
Then define the \textbf{projection} of a leaf on the internal vertices as its last ancestor that lies on the opposite side.
This indeed defines a bijection between the leaves, except the rightmost one, and the internal vertices, that one can easily invert: send each left internal vertex to the rightmost leaf among its descendants and vice versa.
From the convention that the root is a right vertex, it is the projection of the leftmost leaf; however the rightmost leaf that has no projection: we say that the projection ``exits the tree''.

We will study in this work the following two-step algorithm, which will be henceforth referred to as the \textbf{BBJ dynamics}.
First initialise the process with $T_0^{\, \bullet} \coloneqq (T_0, X_0)$ being a single marked vertex (a leaf) and then for any $n \ge 0$, build $T_{n+1}^{\, \bullet} \coloneqq (T_{n+1}, X_{n+1})$ from $T_n^{\, \bullet} \coloneqq (T_n, X_n)$, as follows:
\begin{itemize}
\item First build $T_{n+1}$ by applying the one-step R\'emy growth as in Figure~\ref{fig:remy} left: replace the marked vertex $X_n$ in $T_n^{\, \bullet}$ by a new branchpoint with two offspring: $X_n$ with its entire progeny and a new leaf $L_{n+1}$, placed either on the left or on the right independently with probability $1/2$;

\item Then mark this tree at $X_{n+1}$ chosen as follows: again independently, with probability $1/2$ each, either set $X_{n+1} = L_{n+1}$, or let $X_{n+1}$ be the projection of $L_{n+1}$ on the internal vertices of $T_{n+1}$ as in Figure~\ref{fig:remy} right;
since the rightmost leaf of $T_{n+1}$ has no projection, then in this case we sample $X_{n+1}$ uniformly at random and independently among all vertices of $T_{n+1}$.
\end{itemize}
The last case where the new leaf is the rightmost one and we need to sample the marked vertex uniformly at random will be henceforth called \textbf{resampling steps}. The chain $(T_n,X_n)_{n \ge 0}$ produced this way has the property that at any given time $n\geq 0$ the  marked tree $T_n^{\, \bullet}=(T_n, X_n)$ has the uniform distribution on its size (see Lemma~\ref{lem:indep_tree_failure} below). Figure~\ref{fig:ex_dyn} displays a possible example of the first steps of such a sequence of random marked trees.

\begin{figure}[h!]
\centering
\includegraphics[page=4, width=.95\linewidth]{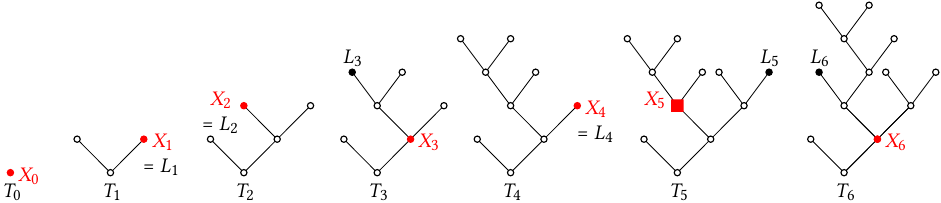}
\caption{A few steps of the BBJ dynamics. 
The current marked vertex $X_n$ is shown in red and the newly added leaf $L_n$ in black, unless it coincides with $X_n$. The fifth step is a resampling step.
}
\label{fig:ex_dyn}
\end{figure}

An important tool in our study will be the process $(S_n)_{n \ge 0}$ recording the height (i.e.~the generation) of the marked vertex in the successive trees. Its scaling limit will enable us to build the continuum analogue of the discrete tree dynamics that we shall now describe.

\subsection{Continuum tree growth}

The \textbf{R\'emy continuous diffusion} will be obtained by taking the scaling limit of the discrete algorithm above. Let us describe informally this continuum random dynamics $(\boldsymbol{\mathcal{T}}^{\, \bullet}_t)_{t \ge 0}$ and refer to Section~\ref{sec:contiuousdynamics} for details.
This dynamics takes value in the space $\mathbb{T}^{\, \bullet}$ of compact real trees\footnote{To help the reader we shall always use bold fonts to denote measured rooted trees, and use the symbol ${}^{\, \bullet}$ when they have an additional marked point.} $(\mathcal{T}, d)$ equipped with a finite measure $\mu$, a root $\rho$, and an extra marked point $x$. This space is endowed with the marked Gromov--Hausdorff--Prokhorov topology, see Section~\ref{sec:preliminaries}.

We introduce a gluing--stretching procedure $\Psi$ that takes as an input such a marked real tree $\boldsymbol{\mathcal{T}}^{\, \bullet}=(\mathcal{T},d,\mu,\rho,x) \in \mathbb{T}^{\, \bullet}$ and a continuous function $g \colon [0, \infty) \to \R$ such that $g(0)=0$ and returns a \emph{random} c\`adl\`ag function $\Psi(\boldsymbol{\mathcal{T}}^{\, \bullet}, g) = (\boldsymbol{\mathcal{T}}^g_t)_{t \ge 0}$ with values in $\mathbb{T}^{\, \bullet}$, with $\boldsymbol{\mathcal{T}}^g_0 = \boldsymbol{\mathcal{T}}^{\, \bullet}$.  
It can be heuristically described as follows, see also Figure~\ref{fig:phig} for an illustration.
We start from the point $x$ in $\mathcal{T}$ and follow the dynamics of a point whose height (distance to the root) variation is given by $g$:
\begin{itemize}
    \item When $g$ decreases, we go down the tree towards the root and, doing so, we \emph{stretch} the distances along the ancestral line by a factor of $2$;
    \item When $g$ increases, we \emph{glue} new pieces of the tree given by the contour of $g$.
\end{itemize}
In particular, when $g$ completes an excursion above its running infimum, this corresponds to the creation of a subtree coded by \textbf{twice} that excursion in the usual sense, which is glued on the tree $\boldsymbol{\mathcal{T}}^{\, \bullet}$ on the ancestral line of its marked point $x$, at a distance from $x$ given by the negative of the minimum of $g$ at this time.
Notice that locally the effect is similar to the one-step R\'emy algorithm of Figure~\ref{fig:remy} left: we stretch the ancestral line and graft new branches.

\begin{figure}[h!]
    \centering
    \includegraphics[page=5, width=.95\linewidth]{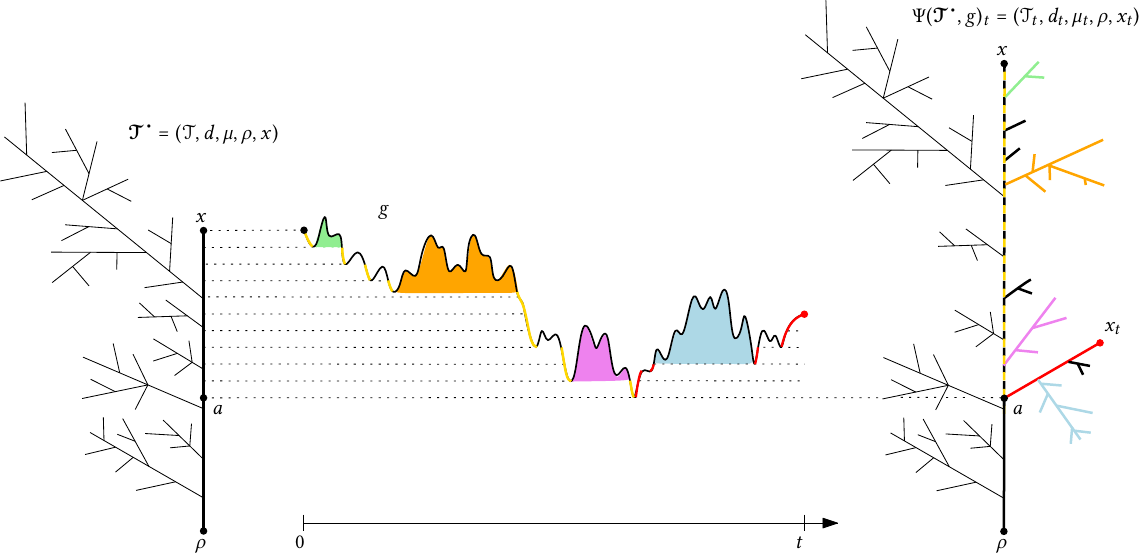}
    \caption{    Illustration of the application $\Psi(\boldsymbol{\mathcal{T}}^{\, \bullet},g)$ that produces the tree on the right at time $t$.
    All the sub-excursions of $g$ above its running minimum create new subtrees (here in different colours) coded by twice the corresponding excursion, and the instants of a new minimum double the distances along the ancestral line of the point $x$ (the yellow parts).
    The new tree $\Psi(\boldsymbol{\mathcal{T}}^{\, \bullet},g)_t$ is marked at $x_t$, corresponding to the tip of $g$ at time $t$, while the root $\rho$ is left unchanged.}
    \label{fig:phig}
\end{figure}

Let $\mathrm{Ht}(x) = d(\rho, x)$ denote the height of $x$ in the initial tree $(\mathcal{T}, d)$.
We shall prove in Section~\ref{sec:contiuousdynamics} that, up to time \[\resamp = \inf\{t > 0 \colon g(t) < -\mathrm{Ht}(x)\},\]
this construction is well-defined and provides a function $(\Psi(\boldsymbol{\mathcal{T}}^{\, \bullet}, g)_t)_{t \in [0, \resamp)}$ which is \emph{deterministic and continuous} in $t$ and which depends in a Lipschitz-continuous way in both the initial tree $\boldsymbol{\mathcal{T}}^{\, \bullet}$ and the function $g$.
In addition, it enjoys the semi-group property: for every $s, t \ge 0$ such that $t+s < \resamp$, letting $g|_{[u,u+v]} \colon w \in [0,v] \mapsto g(u+w)-g(u)$, we have:
\[\Psi(\boldsymbol{\mathcal{T}}^{\, \bullet}, g)_{t+s}
= \Psi(\Psi(\boldsymbol{\mathcal{T}}^{\, \bullet}, g)_{s}, g|_{[s,s+t]})_{t}
= \Psi(\Psi(\boldsymbol{\mathcal{T}}^{\, \bullet}, g)_{t}, g|_{[t,s+t]})_{s}
.\]

This time $\resamp$ is the analogue of the resampling events in the discrete dynamics, when we exit below the root; accordingly, we must now resample the marked point randomly in the current tree. This is where both discontinuity and randomness enter the definition of $\Psi$. To do this, recall that the initial tree $\boldsymbol{\mathcal{T}}^{\, \bullet}$ is endowed with a finite measure $\mu$; the latter evolves then with $g$ to provide a finite measure on $\boldsymbol{\mathcal{T}}^g_{\resamp-}$ that we use to resample the marked point at time $\resamp$.
After resampling at this time $\resamp=\resamp_1$, we then continue the construction as above with the remaining part of the function $g$ up to the next time $\resamp_2$ this situation occurs again, etc.
We shall see that these resampling events do not accumulate (Lemma~\ref{lem:pas_d_accumulation}), so the process is well-defined for every time $t>0$.

Consider as initial tree a single point $\mathbf{o}$ and as driving function $g$ a standard Brownian motion $B = (B_t)_{t \geq 0}$. Then $\resamp=0$ and the previous construction $\Psi(\mathbf{o},B)$ is ill-defined near $t=0$, but the existence of such a process  will be proved in the following theorem by approximations with the BBJ dynamics. For every $n \geq 0$ we let $\boldsymbol{T}_n^{\, \bullet}$ denote the measured rooted and marked tree obtained by equipping the $2n+1$ vertices of the discrete marked binary tree $T_n^{\, \bullet}$ with the graph distance and the counting measure. 
In the rest of the paper, if $\boldsymbol{\mathcal{T}}^{\, \bullet}=(\mathcal{T},d,\mu,\rho,x)$ is any measured rooted marked tree and $\alpha>0$ any positive number, we denote by $\alpha \cdot \boldsymbol{\mathcal{T}}^{\, \bullet}$ the tree
\[\alpha \cdot \boldsymbol{\mathcal{T}}^{\, \bullet} = (\mathcal{T}, \alpha d, \alpha^2 \mu, \rho, x).\]

\begin{thm}[R\'emy's diffusion as the scaling limit of the Bacher--Bodini--Jacquot chain]\label{thm:cv_BBJ_resampling}
The convergence in distribution
\[\Bigl(\frac{1}{\sqrt{2n}} \cdot \boldsymbol{T}_{\lfloor nt \rfloor}^{\, \bullet}\Bigr)_{t \ge 0} \cvloi (\boldsymbol{\mathcal{T}}^{\, \bullet}_t)_{t \ge 0},\]
holds for the Skorokhod topology on $\D( \R_{+}, \mathbb{T}^{\, \bullet})$. The limit process is such that for every $t>s>0$, the tree
$\boldsymbol{\mathcal{T}}_t^{\, \bullet}$ has the same law as 
$\Psi(\boldsymbol{\mathcal{T}}_s^{\, \bullet}, B|_{[s, t]})$ where $B$ is an independent standard Brownian motion. 
\end{thm}
The process $( \boldsymbol{ \mathcal{T}}_{t})_{t \geq 0}$ obtained from $( \boldsymbol{ \mathcal{T}}_{t}^{\, \bullet})_{t \geq 0}$ by forgetting the marked point is continuous: the only source of discontinuities are the discrete resampling instants of the marked point when it collides with the root.
Since $T_n$ is uniformly distributed for every $n$, then Aldous' invariance principle~\eqref{eq:aldous} shows that for each $t>0$, the tree $\boldsymbol{\mathcal{T}}_t$ has the law of the Brownian CRT with mass $t$.

A consequence of Theorem~\ref{thm:cv_BBJ_resampling} and the convergence of trees~\eqref{eq:aldous} is that if $ \boldsymbol{\mathcal{T}}_{2\mathbb{e}}^{\, \bullet} = (\mathcal{T}_{2 \mathbb{e}}, X)$ is a Brownian CRT of mass $1$ with an independent random marked point $X \in \mathcal{T}_{2 \mathbb{e}}$ sampled from its mass measure, and if $B$ is an independent Brownian motion then for any  $\varepsilon >0$,
\begin{equation}\label{eq:infinitesimal}
\Psi(\boldsymbol{\mathcal{T}}_{2\mathbb{e}}^{\, \bullet}, B)_{\varepsilon} 
\enskip\text{has the law of a marked Brownian CRT of mass } 1+ \varepsilon
.\end{equation}
We consider this fact as the infinitesimal version of the continuous R\'emy diffusion on Brownian CRT's. Although it may have a direct proof in the setting of continuum trees (but we do not know how to do it), we deduce it here from the discrete setting by passing to the limit.

In particular, by Brownian scaling, the rescaled dynamics $ \boldsymbol{\mathcal{T}}^{\, \bullet} \mapsto \frac{1}{ \sqrt{1+t}} \cdot \Psi( \boldsymbol{\mathcal{T}}^{\, \bullet}, B)_t$ leaves invariant the law of the marked Brownian CRT $ \boldsymbol{\mathcal{T}}_{2\mathbb{e}}^{\, \bullet} = (\mathcal{T}_{2 \mathbb{e}}, X)$. Even though the law of processes started from two distinct initial conditions may remain mutually singular at every finite time, the rescaled dynamics exhibits an \textbf{asymptotic coupling} in the sense of~\cite{HMS11}: the two processes can be coupled so that their 
distance converges to zero.
This yields in particular the following unique ergodicity property.

\begin{thm}[Unique ergodicity]\label{thm:ergo}
Let $\boldsymbol{\mathcal{T}}^{\, \bullet}$ be any rooted real tree equipped with a mass measure with total mass $1$ and a marked point different from the root. Let $B$ be an independent Brownian motion. Then $\frac{1}{ \sqrt{1+t}} \cdot \Psi(\boldsymbol{\mathcal{T}}^{\, \bullet}, B)_t$ converges in distribution towards $ \boldsymbol{\mathcal{T}}_{2\mathbb{e}}^{\, \bullet} = (\mathcal{T}_{2 \mathbb{e}}, X)$ as $t\to\infty$ in the marked Gromov--Hausdorff--Prokhorov topology.
\end{thm}

\subsubsection*{Diffusions on Brownian CRT's}

Several Markovian dynamics on real trees with the Brownian CRT as invariant measure can be found in the literature, which, interestingly, all arise as scaling limits of discrete tree-valued Markov chains, but are constructed and analysed through different lenses.

The \textbf{prune-and-regraft} dynamics was introduced by Evans,  Pitman, \&~Winter~\cite{EPW06} as a scaling limit of the Aldous--Broder algorithm, see also~\cite{EW06}. 
The so-called \textbf{Aldous' diffusion} has been constructed rigorously by Forman, Pal, Rizzolo, \&~Winkel in the culmination of a series of work~\cite{FPRW23}; see also the work of Lohr, Mytnik, \&~Winter~\cite{LMW20} in a weaker topology. This diffusion is the scaling of the discrete chains on cladograms which consists heuristically to perform R\'emy's one-step move back and forth at i.i.d.~locations in the tree. 
Another diffusion can be obtained by projecting Zambotti's dynamics~\cite{Zam01} on paths through the Duquesne--Le~Gall coding of trees. This dynamics is the scaling of a \textbf{corner flip dynamics} as proved in~\cite{AL15}.

Arguably, our new dynamic is simpler to construct and has the advantage of being completely local. A companion paper~\cite{CCFT26} introduces yet another such local dynamics, with the strong property of being monotone, but whose growth has the disadvantage of evolving in a c\`adl\`ag fashion (on trees, even with no marked point). Our work gives a great latitude to define a growth procedure on continuum trees and we plan to study in the future a continuous version of the present one.

More broadly, understanding Markovian dynamics whose invariant distribution is a canonical probability measure $\mathscr{L}$ is of fundamental importance, as such dynamics can provide valuable information about $\mathscr{L}$ itself, including monotonicity properties, fixed-point characterisations, and tools for proving convergence to equilibrium. In the Gaussian setting, classical examples include Langevin dynamics and the Markovian interpretation of Stein's method. In the context of Yang--Mills theory, the stochastic quantisation program initiated by Parisi \& Wu seeks to construct and study a difficult probability measure $\mathscr{L}$ through a comparatively tractable Markovian dynamics having $\mathscr{L}$ as its invariant measure; see~\cite{CCHS22} for a recent breakthrough in this direction.

In the more modest setting of real trees, one of our aims is to \emph{develop an infinitesimal, or `generator-like', approach to Markovian dynamics on the space of real trees}. More precisely, we seek a way to characterise nice Markov processes through their infinitesimal action in the vein of~\eqref{eq:infinitesimal}. Such an approach might provide a common framework for comparing the various dynamics described above.

\paragraph{Acknowledgments:}  We thank  Alessandra Caraceni, Serte Donderwinkel, William Fleurat, Robin Stephenson, and Adrianus Twigt for stimulating discussions. The first author is supported by ``SuperGrandMa'', the  ERC Consolidator Grant No 101087572. 

\section{Preliminaries on the discrete dynamics}

In this section we study the discrete dynamics $(T_n, X_n)_{n\ge 0}$ defined in the Introduction using the Bacher--Bodini--Jacquot coupling of R\'emy's algorithm. Recall in particular the definition of resampling events when exiting the tree from the rightmost leaf.

First, the fact that for each $n \geq 0$ the random marked tree $T_n^{\, \bullet}= (T_n,X_n)$ is uniformly distributed 
is a direct consequence of the following two lemmas whose proofs are clear from Figure~\ref{fig:remy}.

\begin{lem} \label{lem:remy1}
The one-step R\'emy growth operation of Figure~\ref{fig:remy} left is a bijection between:
\begin{itemize}
    \item binary trees with size $n$ marked at a vertex together with a coin flip, i.e.~triples $(t_n,x_n, \epsilon_n)$ where $x_n$ is any vertex of $t_n$ and $\epsilon_n \in \{\mathrm{left}, \mathrm{right}\}$,
    
    \item binary trees with size $n+1$ marked at a leaf, i.e.~pairs $(t_{n+1}, \ell_{n+1})$ where $\ell_{n+1}$ is a leaf of $t_{n+1}$.
\end{itemize}

\end{lem}

\begin{lem} \label{lem:proj}
Let $t_n$ be a binary tree with size $n$ and $L_n$ a uniform random leaf of $t_n$. Independently and with probability $1/2$ each, let either $X_n = L_n$ or let $X_n$ be the projection of $L_n$ on $t_n$ as in Figure~\ref{fig:remy} right; if $L_n$ is the rightmost leaf, then let $X_n$ be an independent uniform random vertex of $t_n$ instead in the second case.
Then $X_n$ is uniformly distributed among the $2n+1$ vertices of $t_n$.
\end{lem}

In the BBJ chain $(T_n^{\, \bullet})_{n\geq 0}$, the set of vertices of $T_{n+1}$ is naturally identified to those of $T_n$ together with an additional internal vertex and the leaf $L_{n+1}$, hence if $x \in T_n$, we will abuse notation and still write $x \in T_{n+1}$ for the corresponding vertex in $T_{n+1}$. Beware, if $x,y \in T_n$, their distance in $T_{n+1}$ might differ from that  in $T_n$ (it may only increase, though).

\subsection{On the tree and resampling steps}

The combination of the two previous lemmas shows that the marked tree produced by the algorithm at time $n$ has the uniform distribution. Quite surprisingly, this marked tree is independent of the instants of the resampling steps that have occurred so far, which are also independent of each other.
As a consequence, not only is $(T_n, X_n)$ uniformly distributed for any fixed $n$, but this is also the case at the resampling instants.

For every $k \ge 1$, let $I_k = 1$ if the $k$'th step is a resampling step and $I_k=0$ otherwise.

\begin{lem}[Independence of the final tree and resampling steps]\label{lem:indep_tree_failure}
For every $n \ge 1$, the variables $(T_n^{\, \bullet}, I_1, \dots, I_n)$ are independent, 
each $I_k$ has the Bernoulli distribution with parameter $1/(2k+2)$ respectively,
and $T_n^{\, \bullet}=(T_n, X_n)$ has the uniform distribution on marked binary trees with size $n$.
\end{lem}

\begin{proof}
The key observation is that the final marked tree $(t_n, x_n)$ and the indices of resampling steps determine uniquely the entire sequence $(t_1, x_1), \dots, (t_n, x_n)$. 
Indeed, let us consider the dynamics run backward. For any $k=n, \dots, 1$, if the $k$'th step is not a resampling step, then given $(t_k, x_k)$ one can recover the intermediate marked leaf of $t_k$: it is the right-most leaf among the descendants of $x_k$ if the latter is a left child and vice versa. If the $k$'th step is a resampling step, then this means that the marked leaf was the rightmost leaf of $t_k$. In each case one thus recovers the marked leaf and we can then invert the growth procedure to recover $(t_{k-1}, x_{k-1})$.

At each step, one makes a uniform random choice among $4$ possibilities and independently of the rest, and in case of resampling at step $j$, one extra choice of a vertex with $2j+1$ total possibilities is needed.
Then the probability to produce $(t_n, x_n)$ and to resample exactly at times in a given subset $J \subset \{1, \dots, n\}$, or equivalently to produce the unique corresponding sequence $(t_1, x_1), \dots, (t_n, x_n)$ with resampling precisely at the steps with index in $J$,
equals
\[4^{-n} \prod_{j \in J} \frac{1}{2j+1}.\]
Let $\mathrm{Cat}(n) = \frac{1}{n+1} \binom{2n}{n}$ denote the $n$'th Catalan numbers, which counts the number of rooted binary plane trees with size $n$.
After some simple manipulations, we obtain:
\[4^{-n} \prod_{j \in J} \frac{1}{2j+1} = \frac{1}{(2n+1) \mathrm{Cat}(n)} \cdot \prod_{j \in J} \frac{1}{2j+2} \cdot \prod_{j \in J^c} \Bigl(1 - \frac{1}{2j+2}\Bigr)
,\]
which proves our claim.
\end{proof}

Let us mention that the algorithm originally designed in~\cite{BBJ17} actually consisted, when exiting the tree from the rightmost leaf, to start all over from a single vertex, again and again, until for the first time one reaches a binary tree with size $n$ without any exit.
As argued there, this is (nearly) optimal in terms of random number generation, informally because resampling steps mostly occur when the tree is small. 
Compared to the present dynamics, this is equivalent to conditioning not to require resampling before time $n$ and the previous lemma shows that this indeed produces a uniform random binary tree with size $n$. Note that the same holds if one conditions instead the dynamics to require resampling for the first time precisely at time $n$.

In these pages, we shall only use the algorithm with resampling a uniform random vertex when exiting the tree.
The reason comes from the observation that resampling at time $n$, although independent of the tree at time $n$, is however not independent of the tree at time $n-1$.
In particular, if one conditions the dynamics to not require resampling before time $n$, then Lemma~\ref{lem:indep_tree_failure} does show that $T_n$ has the uniform distribution, however $T_k$ for $k<n$ does \emph{not} have the uniform distribution and in the scaling limit, this does \emph{not} yield Brownian CRT's before time $1$.

\subsection{The dynamics along the spine between two resampling steps}
\label{subsec:description_dynam}

Let us now describe precisely how the marked tree is transformed between two consecutive resampling steps.
Equivalently, instead of starting the dynamics from a single vertex, let it start from an arbitrary marked binary tree $(t,x)$, and let us run it until the first resampling step, at time $\Resamp = \inf\{k \ge 1 \colon I_k = 1\}$.

The statements and proofs of this section should become intuitive once one has a clear view of how the dynamics work.
A moment of thought leads to the following description, which is represented in Figure~\ref{fig:spine_snake}.
At the first step, the vertex $x$ is replaced by a new internal vertex; then $x$ and its progeny is pushed up by one generation, and a new leaf is created next to it. Suppose that this leaf is placed on opposite side of its parent and that we keep this leaf as the new marked vertex (each has a probability $1/2$ of occurring). Then after that, we create an entire new subtree rooted at this leaf before exiting it; at this moment, the mark lands precisely where $x$ was at time $0$. 
If at the first step we do not keep the mark on the leaf but instead we project on its parent, then the only difference is that the new subtree rooted at the leaf is reduced to just that leaf.

After that, the dynamics then continues in the same way:
We replace the current marked vertex by a new internal one, move everything above it up by one generation, and create a new leaf. If this leaf is placed on opposite side of its parent again, then again we next create an entire new subtree before reaching back again the original position of $x$. 

Hence, after a random geometric number, possibly zero, of subtrees created on opposite side of $x$, we replace the current vertex sitting at the original position of $x$ by a new internal vertex and place the leaf on the \emph{same side} as its parent. Now again we create after that an entire subtree rooted at that leaf (possibly reduced to just that leaf) before exiting it. The difference with the previous case is that, when exiting it this subtree, we do not land on the original position of $x$, but on its last ancestor which lies on opposite side, which we call a \emph{turning vertex}.

After that, the same story takes place at this turning vertex: we first create a random geometric number, possibly zero, of subtrees on the opposite side (in red in Figure~\ref{fig:spine_snake}), then exactly one subtree on the same side (in green in Figure~\ref{fig:spine_snake}), then move down to the next turning vertex, etc.~until we leave the entire tree, at time $\Resamp$.
Notice that a no time before $\Resamp$ does the dynamics  enter the subtrees that are attached on the ancestral line of $x$ in the original tree $(t,x)$, and thus leaves them completely intact.

Let us record for every $n \ge 0$:
\[S_n \enskip\text{the height of $X_n$ in $T_n$.}\]
We are going to see that this sequence gives a precise description of the process $(T_n, X_n)_{n\ge 0}$, similar to the role played by the function $g$ in the continuous dynamics informally described in the Introduction.
For the moment, from the previous description, we see that for every time $0 \le k < \Resamp$,
\[X_k \text{ is an ancestor of } x \text{ or } X_k=x
\iff
S_{k} = \min_{[0,k]} S.\]
In addition, this ancestor $X_k$ at a weak negative record of $S$ is either $x$ or a turning vertex along its ancestral line. Precisely, it is $x$ whenever $S_{k} = \min_{[0,k]} S = S_0$, and it is the $i$'th turning vertex (from top to bottom) when $S$ has made a total of $i$ strict negative records by time $k$.
See again Figure~\ref{fig:spine_snake} for a pictorial description.

\begin{figure}[h!]
    \centering
    \includegraphics[page=8, width=.9\linewidth]{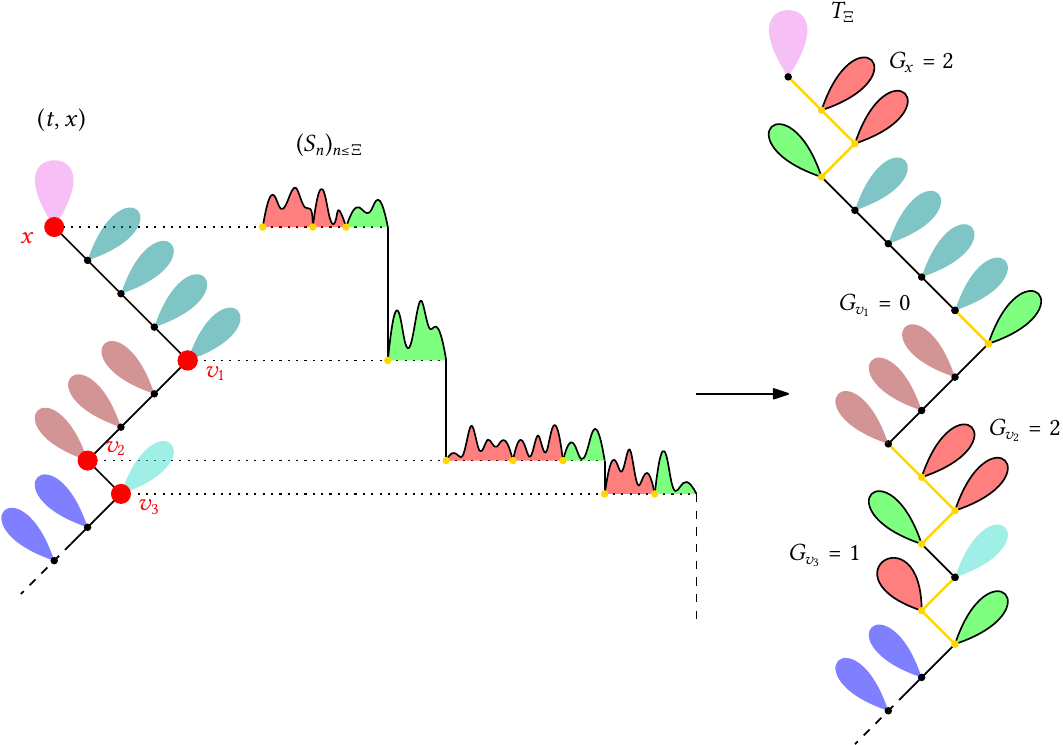}
    \caption{
        Left: the initial marked tree $(t,x)$ with the turning vertices on its ancestral lines in large red.
    Middle: the height $(S_n)_n$ of the marked vertex in the dynamics. In gold are shown the instants where it sits at is running minimum.
    Right: the tree $T_\Resamp$ created by the dynamics at the first resampling step $\Resamp$. It has only visited the vertex $x$ and the turning vertices, where it has created subtrees coded by the excursions of $S$ above its minimum. 
    All coloured trees on the left are unchanged in the right picture.
    }
    \label{fig:spine_snake}
\end{figure}

In addition to being a natural quantity, let us motivate the study of $(S_n)_{n \ge 0}$ with the following observation:
at each step $j \ge 1$, the dynamics creates a new internal vertex and a leaf, and the internal vertex replaces $X_{j-1}$ so the height of this internal vertex, when created, is $S_{j-1}$, while the new leaf is placed at height $S_{j-1}+1$.
We shall see in a second step how these heights evolve in the subsequent trees $(T_{n+j})_{n \ge 1}$.

Before that, let us label these new vertices: $2j$ refers to the internal one and $2j+1$ the leaf, both created at step $j$. 
In this way, if the initial tree $t$ is composed of a single vertex labelled $1$, then this provides a labelling of all the vertices of $T_n$, from $1$ to $2n+1$, that is consistent as $n$ varies, and where internal vertices receive an even label and leaves receive an odd label. This labelling corresponds to the colour scheme in the simulations from Figure~\ref{fig:simu}.
The previous description of the ancestral lines yields the following result, for all initial trees $(t,x)$.

\begin{lem}\label{lem:ancetre_commun}
Let $1 \le i < j \le n$.
On the event that $\Resamp \ge n$, the most recent common ancestor of the internal vertices $2i$ and $2j$ in the tree $T_n$ is the internal vertex $2 (i \wedge_n j)$, where
\[i \wedge_n j = \sup\{k \in \{i, \dots, j\} \colon S_{k-1} = \min_{[i-1,k-1]} S\}.\]
\end{lem}

\begin{proof}
Recall that at step $i$, the vertex $2i$ replaces the previous marked vertex $X_{i-1}$, at height $S_{i-1}$. At time $(i \wedge_n j) - 1$, since $S$ is at its minimum since time $i-1$, then the marked vertex $X_{(i \wedge_n j) - 1}$ is an ancestor of $2i$. Then at the next step we replace this marked vertex by $2(i \wedge_n j)$, and up to time $j$, the path $S$ remains strictly above this value, meaning that it creates a subtree rooted at $2(i \wedge_n j)$. This shows our claim when $n=j$. Between time $j$ and $\Resamp$, the dynamics never modifies the subtrees attached along the ancestral line of $j$, so the vertex $2(i \wedge_n j)$ will remain the most recent common ancestor of $2i$ and $2j$.
\end{proof}

\subsection{Height of the marked vertex}

Let us now describe the law of the sequence $(S_n)_{n \ge 0}$ recording the height of the successive marked vertices, when again we let the dynamics start from an arbitrary marked binary tree $(t,x)$.
The construction of the algorithm suggests that it does not evolve as a random walk since downward steps, in the case of projection, depend on the current tree structure. 
However besides this, it does exhibit independent increments.

\begin{defi}[coding the ancestral lineage]\label{def:snake}
Let $(t,x)$ be a binary tree with a marked vertex at height $h \geq 1$. For every $i=0, \dots, h-1$, let $a_i = 0$ if the ancestor of $x$ at height $h-i$ lies on the same side as its parent, and $a_i = 1$ if it lies on the opposite side.
Let $A = \sum_{i=0}^{h-1} a_i$.
For every $1 \le k \le  A$, let $p_{k}= 1+\inf\{j \ge 0 \colon \sum_{i=0}^j a_i \ge k\}$ and then $\ell_k = p_{k}-p_{k-1}$, where $p_0=0$.
Let finally $\ell_{A+1} = h+1-p_{A}$.
\end{defi}

Since the root lies on the right by convention, the sequence $(a_i)_{0 \le i \le h-1}$ allows to recover the left/right position of all the ancestors of $x$.
Those with $a_i=1$ are the children of the turning vertices of Figure~\ref{fig:spine_snake} from top to bottom.
The quantities $(\ell_k)_{1 \le k \le A+1}$ count the length of the intervals between these turning vertices, if we count $x$ and the root as turning vertices.

Let us introduce some notation for the next result.
Let $G$ be a positive geometric random variable, whose law is $\P(G=k)=2^{-k}$ for every $k \ge 1$.
Given the tree $(t,x)$, let $W$ denote a random walk with step distribution $2-G \in \Z_{\le 1}$ started at $W_0=h$, the height of $x$ in $t$. Then modify this path $W$ to create a path $W^x$ as follows: for every $1 \le k \le A+1$, change the value at the $k$'th time $n$ such that $W_n < \inf_{[0,n-1]} W$ and instead set the value of $W^x$ at this time to be equal to $h-p_k$. All the other increments are unchanged. The path $W^x$ then stops after its $A+1$'st strict negative record, by reaching the value $h-p_1-\ell_{A+1} = -1$ for the first time.

Let us run the dynamics starting from $(t, x)$ and until the first resampling step $\Resamp$, but without performing the resampling of the vertex at the end and simply declare $S_{\Resamp} = -1$.

\begin{prop}[The law of height process]
\label{prop:law_height_marked_point}
The path $(S_n)_{n \le \Resamp}$ has the same law as $W^x$.
\end{prop}

\begin{proof}

Let us imagine that the $a_i$'s are unknown.
When performing the first one-step R\'emy growth starting from the tree $(t,x)$, the new leaf is placed one generation above $x$; let us not reveal right away if it is placed on the same side as $x$ or not but instead append this information to the previous unknown vector. Note that this entry is $0$ or $1$ with probability $1/2$ and is independent of the rest.

With probability $1/2$ we keep this leaf as the next marked vertex, in which case $S$ increases by $1$, whatever the side of this leaf and the unknown vector.
With a probability $1/2$ we project the mark from this leaf and now we start revealing our unknown vector, from the leaf we just added to the root, until we find an entry equal to $1$.
In this case $S$ makes a nonpositive increment, whose value equals $1$ minus the number of entries we revealed before the first $1$.
Notice that when we stop and find the ancestor that will be the next marked vertex, we have not revealed the $a_i$'s corresponding to this vertex and its ancestors in our vector.
We may then iterate this argument.

The description of the law of $S$ follows since, during the construction, the $0$'s and $1$'s that we add are i.i.d.~Bernoulli distributed with probability $1/2$ and independent of the rest, but the original $a_i$'s are deterministic, and this corresponds to the modification $W^x$ described above.
\end{proof}

Let us deduce an invariance principle from this description. Informally, we say that a vertex is \emph{good} if it admits the same law of large numbers along its spine as in a uniform random spine.

\begin{defi}[Good vertices]\label{def:good}
Let us fix once and for all a value $\mathfrak{s} \in (0,1/6)$.
Let $(t,x)$ be a binary tree with a marked vertex at height $h \geq 1$. 
We say that $(t,x)$ is \emph{$n$-good} if
\[h \leq n^{\mathfrak{s}+1/2}
\qquad\text{and}\qquad
\sup_{t \in [0,1]} \Big|\frac{h}{2} t -  \sum_{i=1}^{\lfloor h t\rfloor} a_i\Big| \leq n^{\mathfrak{s}+1/4},\]
where the $a_i$'s are as in Definition~\ref{def:snake}.
A non-marked binary tree $t$ is \emph{$n$-good} when $(t,x)$ is $n$-good for all $x \in t$.
\end{defi}

This property will be in force for all points considered in our dynamics.

\begin{lem}[All good]\label{lem:all_good}
The following events occur with a probability tending to $1$ as $n \to \infty$:
\begin{enumerate}
\item $T_0^{\, \bullet}, \ldots, T_n^{\, \bullet}$ are all $n$-good when $T_0^{\, \bullet}$ is a single vertex;
\item $T_0^{\, \bullet}, \ldots, T_{\Resamp-1}^{\, \bullet}$ are all $\max(n,\Resamp)$-good when $T_0^{\, \bullet}=(t,x)$ is $n$-good.
\item $T_{\Resamp}$ is $\max(n,\Resamp)$-good when the non-marked tree $T_0$ is $n$-good.
\end{enumerate}
\end{lem}

\begin{proof}
\begin{enumerate}
\item Fix $k \le n$; we know that $T_k^{\, \bullet}$ is uniformly distributed, so the $a_i$'s corresponding to $X_k$ are i.i.d.~Bernoulli random variables with parameter $1/2$. We can then apply a large deviation result for the binomial random walk $(\mathrm{Bin}_n)_{n \ge 1}$, 
namely for every $N \ge 1$ and $t \ge 0$,
\[\P(\max_{1 \le i \le N} |\mathrm{Bin}_i-i/2| \ge Nt) \le 2 \e^{-2 Nt^2},\]
see e.g.~\cite{McD98} at the very end of Section $2$ there.
In addition, the height $H_k$ of $X_k$ divided by $\sqrt{k}$ has sub-Gaussian tails (for binary trees, this can be traced back to~\cite[Theorem~1.3]{FGOR93}) so the probability that $X_k$ is not $n$-good decays as the exponential of some positive power of $n$. We conclude from a union bound.

\item Fix $k < \Resamp$. The marked vertex $X_k$ in $T_k^{\, \bullet}$ belongs to a subtree made only of vertices created since $T_0$ (in either green or red in Figure~\ref{fig:spine_snake}) and that is grafted on the ancestral line of $X_0$ in $T_0$, at a turning vertex $V_k$ (or $X_k$ is this turning vertex). The ancestral line of $V_k$ has not yet changed since time $0$ and is therefore $n$-good.
As for the new subtree grafted above $V_k$ and marked at $X_k$, the previous argument shows that it is $n_k$-good where $n_k \le \Resamp$ is its number of internal vertices.
More precisely, conditionally on $n_k$, this subtree has the law of $T_{n_k}^{\, \bullet}$ started from a single vertex and conditioned of the first resampling instant being greater than $n_k$. According to Lemma~\ref{lem:indep_tree_failure}, this tree is uniformly distributed so the argument above applies.

\item Recall that $T_{\Resamp}^{\, \bullet}$ is marked at an independent uniform random vertex. Due to the second item, if the initial marked vertex is $n$-good and the random marked vertex at time $\Resamp$ is either of its ancestors or belongs to one of the new subtrees that we created, then $T_{\Resamp}^{\, \bullet}$ is $\max(n,\Resamp)$-good. If all vertices of $T_0$ are $n$-good, then $T_{\Resamp}^{\, \bullet}$ is indeed $\max(n,\Resamp)$-good.
\vspace{-\baselineskip}
\end{enumerate}
\end{proof}

In the next proposition, for every $n \ge 1$, we let the dynamics start from a marked binary tree $(t^{(n)}, x^{(n)})$, not necessarily with $n$ internal vertices, and where $x^{(n)}$ has height $h_n$.
Let $\Resamp^{(n)}$ denote the first resampling time and $(S^{(n)}_k)_{k < \Resamp^{(n)}}$ record the height of the marked vertex before the first resampling step, started from $S^{(n)}_0 = h_n$.
For every $h>0$, let $B^h$ denote a Brownian motion started from $B^h_0=h$ and let $\resamp^h = \inf\{t>0 \colon B^h_t=0\}$ denote the first hitting time of $0$.

\begin{prop}[Quenched Brownian limit]\label{prop:cv_hauteur_entre_resampling}
Suppose that the trees $(t^{(n)},x^{(n)})$ are all $n$-good and that 
\[\frac{h_n}{\sqrt{2n}} \cv h>0.\]
Then the following convergence in distribution holds:
\begin{equation}\label{eq:cv_hauteur_brownien}
\Bigl(\frac{\Resamp^{(n)}}{n}, \Bigl(\frac{1}{\sqrt{2n}} S^{(n)}_{\lfloor n t \rfloor \wedge (\Resamp^{(n)}-1)}\Bigr)_{t \ge 0}\Bigr)
\cvloi
(\resamp^h, (B^h_{t \wedge \resamp^h})_{t \ge 0})
.\end{equation}
\end{prop}

\begin{proof}
If the $a_i^{(n)}$'s were random i.i.d.~Bernoulli with parameter $1/2$ independent of the rest, then according to Proposition~\ref{prop:law_height_marked_point}, the path $S^{(n)}$ would just be a random walk started from $W_0=h_n$ and with step distribution $2-G$, stopped when hitting $\Z_{<0}$ and the convergence~\eqref{eq:cv_hauteur_brownien} would then follow since $2-G$ has mean $0$ and variance $2$.
In addition, in this case our assumption on the $a_i^{(n)}$'s holds by the Law of Large Numbers.

Now to derive a quenched version when the $a_i^{(n)}$'s are given,  Proposition~\ref{prop:law_height_marked_point} shows that the comparison with this random walk holds for the excursions of $S$ above its past minimum. The $a_i^{(n)}$'s only encode the negative overshoots $S^{(n)}_k - \min_{[0,k-1]} S^{(n)}$ at the instants of negative strict records. Precisely, the $j$'th such non-zero quantity is the length $\ell_j^{(n)}$ of the interval between two $1$'s in the sequence $(a_i^{(n)})_{i=0, \dots, h_n-1}$.
The $n$-good assumption implies
\[\sup_{t \in [0,1]} \frac{1}{h_n} \Bigl|S^{(n)}_{\lfloor n t \rfloor \wedge \Resamp^{(n)}} - \widetilde{S}\vphantom{S}^{(n)}_{\lfloor n t \rfloor \wedge \widetilde{\resamp}\vphantom{\Resamp}^{(n)}}\Bigr| \cvproba 0,\]
where $\widetilde{S}\vphantom{S}^{(n)}$ is obtained from ${S}\vphantom{S}^{(n)}$ by replacing the deterministic $a_i^{(n)}$'s by random i.i.d.~Bernoulli with parameter $1/2$ independent of the rest, but keeping the same increments otherwise, and $\widetilde{\resamp}\vphantom{\Resamp}^{(n)}$ is its hitting time of $\Z_{<0}$. We conclude from that case.
\end{proof}

\subsection{Stretching and gluing between two resampling steps}

In this section, we run again the BBJ dynamics $(T_k^{\, \bullet})_{k\geq 0}$ starting from an arbitrary marked binary tree $T_0^{\, \bullet} = (t,x)$. Recall that the vertices of a tree can naturally be identified with some vertices in the subsequent trees. We then consider here the sequence of the height of the initial marked vertex $x$ in the successive trees, which we denote by $(\mathrm{Ht}_{T_n}(x))_{n \ge 0}$. Recall that  $(S_n)_{n \ge 0}$ follows the height of the marked vertex and $\Resamp$ denotes the first resampling step.

\begin{lem}[Stretching]\label{lem:stretching_discret}
Fix $n \ge 1$ and suppose that the initial tree $(t,x)$ is $n$-good. Then the height of $x$ in the tree at time $n \wedge \Resamp$ satisfies:
\[\P\bigl( \bigl|\mathrm{Ht}_{T_{n\wedge \Resamp}}(x) - \mathrm{Ht}_t(x)  - (S_0 -\min_{[0,n \wedge \Resamp - 1]} S) \bigr| \geq n^{1/2 - \mathfrak{s}} \bigr)
\le 2 \exp\Bigl(- \frac{1}{32} n^{1/2 - 3 \mathfrak{s}}\Bigr).\]
\end{lem}

\begin{proof}
The vertex $x$ only moves, and up by $1$, at a step $i \ge 1$ when the marked vertex $X_{i-1}$ is $x$ or an ancestor of $x$ (necessarily a turning vertex). This occurs precisely when the path $S$ lies at its running infimum. Hence
\[\mathrm{Ht}_{T_{n\wedge \Resamp}}(x) - \mathrm{Ht}_t(x) = \#\{k \in \{0, \dots, n\wedge \Resamp-1\} \colon S_k = \min_{[0,k]} S\}.\]
Our claim thus reduces to studying the local time at the minimum of the path $S$ and controlling the error
between the previous display and $S_0 -\min_{[0,n \wedge \Resamp-1]} S$.
Let us simplify the notation and set:
\begin{itemize}
    \item $\Delta_n S = S_0 -\min_{[0,n \wedge \Resamp-1]} S$ the lower deviation of $S$,
    \item $N_n = \#\{k \in \{0, \dots, n\wedge \Resamp-1\} \colon S_k = \min_{[0,k]} S\}$ the number of weak records,
    \item $N^\ast_n = \#\{k \in \{0, \dots, n\wedge \Resamp-1\} \colon S_k < \min_{[0,k-1]} S\}$ the number of strict records.
\end{itemize}
We thus aim to upper bound
\[|\Delta_n S - N_n| \le |\Delta_n S - 2 N^\ast_n| + |N_n - 2 N^\ast_n|.\]

To this end, first notice that the values of the strict records are encoded in the sequence $(a_i)_{0 \le i \le \mathrm{Ht}_t(x)-1}$ describing the ancestral line of $x$. Precisely, the absolute size of the strict descending record equal the length of the intervals in this sequence between two consecutive $1$'s. 
Definition~\ref{def:good} of $n$-good implies that
\[|N^\ast_n - \frac{1}{2} \Delta_n S| \le n^{\mathfrak{s}+1/4}.\]

Next recall from the proof of Proposition~\ref{prop:law_height_marked_point} that outside the strict negative records, the path $S$ behaves as the random walk whose increments are i.i.d.~with the same law as $2-G$ where $G$ has the geometric distribution with parameter $1/2$. By lack of memory of this law, a new negative record is strict or weak with probability $1/2$ independently of the rest. 
Therefore, conditionally on the number $N^\ast_n$ of strict records, the error
\[|N_n - 2 N^\ast_n|,\]
has the law of the absolute value of the sum of $N^\ast_n$ independent i.i.d.~random variables with law $G-2$, and this has sub-Gaussian tails.
This is easily obtained by exploiting the explicit Laplace transform 

\[\log \E[\e^{\lambda (G-2)}]
= - 2\lambda - \log(2 \e^{-\lambda}-1)
\le 2\lambda^2
,\]
for any $|\lambda| < 1/10$.
\end{proof}

\subsubsection*{Removing the leaves}

The walk $S$ describes the height of the marked vertices, which is also the height of the new internal vertices at the time they are created. 
Lemma~\ref{lem:stretching_discret} then shows how this height changes as the dynamics continues.
When dealing with the sole large scale geometry of a tree, leaves do not play any role since they lie at distance $1$ from their parent. 
However leaves play an important role in the resampling and cannot be neglected anymore since an internal vertex can have $0$, $1$, or $2$ offspring that are leaves.
The next lemma controls the graph distance between the leaf $2i+1$ and the internal vertex $2i$ that are created at the same step $i$, at the time $\Resamp$ of the first resampling step.

\begin{lem}\label{lem:distance_interne_feuille}
For every initial tree $(t,x)$, for every $i \ge 1$, the graph distance at the first resampling step after $i$ between the internal vertex $2i$ and the leaf $2i+1$ has the geometric distribution with parameter $1/2$.
\end{lem}

\begin{proof}
This follows from the description of the dynamics at the beginning of Section~\ref{subsec:description_dynam}.
At step $i$, we replace the current marked vertex $X_{i-1}$ with the new internal vertex $2i$ with two offspring: $X_{i-1}$ and the leaf $2i+1$ and then either we mark $X_i=2i+1$ or we project the mark on an ancestor, with probability $1/2$ each. In the latter case, even if this ancestor is $2i$, the dynamics until the next resampling step will not affect the subtree formed by $2i$ and its descendants, so in particular $2i$ and $2i+1$ will stay at distance $1$ from each other.
If, on the contrary, we mark $X_i=2i+1$, then the dynamics will create first an entire subtree above this point, before exiting it and reaching $2i$ or one of its ancestors. In this case, the distance between $2i$ and $2i+1$ has then the same law as $1$ plus the height of $2i+1$ in this subtree. We claim that this height in the second case is geometrically distributed, which yields our claim.

Indeed, if $X_i=2i+1$ is the new leaf, then after that, the height of this vertex grows by $1$ exactly when the current marked vertex $X_j$ is at the position  where $X_i$ was, and all that matters is, at these moments, if the new leaf is placed on opposite side of its parent, in which case we will eventually reach back this position, or on the same side, in which case we will eventually exit the subtree without modifying further the height of the vertex $2i+1$.
Since the probability to place the leaf on the left or on the right each time we are at the root is $1/2$, independently of the rest, the final height of $2i+1$ in this subtree is indeed geometrically distributed.
\end{proof}

\section{Scaling limit of the tree dynamics}
\label{sec:contiuousdynamics}

In this section, we shall make formal the construction of the continuum dynamics described in the Introduction.
This will be the analogue of the coding by the height $(S_n)_n$ of the marked vertex, using a continuous function $g$.
We shall then prove, as stated in Theorem~\ref{thm:cv_BBJ_resampling}, that this dynamics is the scaling limit of the discrete one, when the function $g$ is the Brownian motion, which appears as the scaling limit of $(S_n)_n$ in~\eqref{eq:cv_hauteur_brownien}.

\subsection{Preliminaries on continuum trees}

\label{sec:preliminaries}

Let us recall in this subsection the classical setting of continuum trees and refer to e.g.~\cite{LG05} for a gentle introduction.
A real tree, or $\R$-tree, or continuum tree, is a geodesic metric space $(\mathcal{T}, d)$ without any cycle.
We shall denote by $\llbracket a, b \rrbracket$ the geodesic path between $a$ and $b$.
Continuum trees considered in this  work will be compact, equipped with a Borel finite measure $\mu$ called the \emph{mass measure}, and a distinguished element $\rho$ called the \emph{root}.
We write $\boldsymbol{\mathcal{T}}$ for such a space $(\mathcal{T}, d, \mu, \rho)$, which we call a rooted measured tree, or sometimes simply a tree.
We shall also often distinguish a second element $x$ of the tree (which may or not coincide with $\rho$), and we shall then write $\boldsymbol{\mathcal{T}}^{\, \bullet} = (\boldsymbol{\mathcal{T}}, x) = (\mathcal{T}, d, \mu, \rho, x)$, which we call a marked and rooted measured tree, or simply a marked tree.
This point $x$ will play the role of the marked point where the infinitesimal growth takes place.
We shall denote by $\mathrm{Ht}(z) = d(\rho, z)$ the height of any element $z \in \mathcal{T}$.

The space of such marked or non-marked trees can be equipped with a pseudo-distance, which is a variant of the Gromov--Hausdorff--Prokhorov pseudo-distance on measured compact metric spaces, taking into account the, one or two, distinguished points, as considered in e.g.~\cite[Section~2]{ABBGM17}.
A correspondence between two sets $A$ and $B$ is a subset $C \subset A \times B$ of pairs such that for every $a \in A$, there exists $b \in B$ such that $(a,b) \in C$ and conversely for every $b \in B$, there exists $a \in A$ such that $(a,b) \in C$.
Then the distance $d_{\mathrm{GHP}}^{\, \bullet}(\boldsymbol{\mathcal{T}}^{\, \bullet}_1, \boldsymbol{\mathcal{T}}^{\, \bullet}_2)$ between two marked trees $\boldsymbol{\mathcal{T}}^{\, \bullet}_1 = (\mathcal{T}_1, d_1, \mu_1, \rho_1, x_1)$ and $\boldsymbol{\mathcal{T}}^{\, \bullet}_2 = (\mathcal{T}_2, d_2, \mu_2, \rho_2, x_2)$ is the infimum of those $\lambda > 0$ such that:
\begin{itemize}
\item There exists a correspondence $\mathcal{C}$ between $\mathcal{T}_1$ and $\mathcal{T}_2$ that contains both pairs $(\rho_1, \rho_2)$ and $(x_1, x_2)$ and that satisfies $|d_1(a_1, b_1) - d_2(a_2, b_2)| \le 2 \lambda$ for all pairs $(a_1, a_2), (b_1, b_2) \in \mathcal{C}$;
\item There exists a measure $\nu$ on $\mathcal{T}_1 \times \mathcal{T}_2$ such that if $\nu_i$ is the projection of $\nu$ on $\mathcal{T}_i$, then $\|\nu_1-\mu_1\| + \|\nu_2-\mu_2\| + \nu(\mathcal{C}^c) \le \lambda$, where $\|\cdot\|$ denotes the total variation norm.
\end{itemize}

It can be checked that this function $d_{\mathrm{GHP}}^{\, \bullet}(\cdot, \cdot)$ defines a pseudo-distance on marked trees.
In addition two trees lie at distance zero from each other exactly when there exists an isometric map from one to the other that fixes the root and the marked point, as well as the measure.
We let finally $\mathbb T^{\, \bullet}$ denote the corresponding quotient space, which is separable and complete when equipped with the distance $d_{\mathrm{GHP}}^{\, \bullet}$.
We refer again to~\cite{ABBGM17} and references therein for details.

When dealing with rooted but not marked trees, this definition is adapted in the obvious way by only requiring that a correspondence contains the roots $(\rho_1, \rho_2)$, and we let $(\mathbb T, d_{\mathrm{GHP}})$ denote the corresponding (again, separable and complete) quotient space.

A usual simple way to construct such continuum trees consists in taking a continuous function $h$ from a compact interval $[0, \zeta]$ to $\R$, and considering the pseudo-distance:
\begin{equation}\label{eq:def_distance_arbre_fct}
d_h(s,t) = d_h(t, s) = h(s) + h(t) - 2 \min_{[s,t]} h
,\end{equation}
for every $0 \le s \le t \le \zeta$.
If one identifies points of $[0,\zeta]$ that lie at distance $0$, then one obtains a compact metric space, which in this case is a continuum tree~\cite[Theorem~2.1]{DLG05}. We shall call this tree $(\mathcal{T}_h, d_h)$ the one \emph{coded by $h$}. The mass measure $\mu_h$ in this case is  the image of the Lebesgue measure on $[0,\zeta]$ by the canonical projection, and the tree is naturally rooted at the image of  $\mathrm{argmin}~h$ by this projection. 
The mapping $h \mapsto \mathcal{T}_h$ is Lipschitz-continuous for the uniform topology on continuous functions and the distance $d_{\mathrm{GHP}}$, see~\cite[Lemma~2.3]{DLG05} for the GH distance and~\cite[Proposition~2.10]{ADH13} for the GHP distance (with no distinguished root, but the generalisation is straightforward).

Finally, the Brownian Continuum Random Tree, or here simply Brownian tree or CRT, is the random tree $\boldsymbol{\mathcal{T}}_{2\mathbb e}$ coded in the previous sense by \textbf{twice} the standard Brownian excursion $\mathbb e$. Note that its mass measure is a probability measure, with total mass $1$.
We shall also consider CRT's with mass $m>0$, which are coded in the same way by a Brownian excursion with duration $m$.
The scaling property of Brownian motion shows that if $\boldsymbol{\mathcal{T}} = (\mathcal{T}, d, \mu, \rho)$ is such a CRT with mass $m$, then
\[\frac{1}{\sqrt{m}} \cdot \boldsymbol{\mathcal{T}} = (\mathcal{T}, m^{-1/2} d, m^{-1} \mu, \rho)\]
has the same law as $\boldsymbol{\mathcal{T}}_{2\mathbb e}$.
We shall also consider marked versions of such trees, in which case the extra marked point $x$ will be an independent point sampled from the mass measure, that is, simply the image in the tree of an independent uniform random time.

\subsection{Gluing and stretching}

Let us now make formal the continuous tree growth presented in the Introduction, which informally starts from a marked tree and follows the evolution of a point whose height variation is given by the function $g$:
\begin{itemize}
    \item When $g$ decreases, we go down the tree towards the root and, doing so, we \emph{stretch} the distances along the ancestral line of the current point by a factor of $2$;
    \item When $g$ increases, we \emph{graft} new pieces of the tree given by the contour of $g$.
\end{itemize}
We start with the case where there is no `resampling', i.e.~when the height of the point stays nonnegative.
Let us refer to Figure~\ref{fig:def_Phi} for a graphical representation of the construction.

\subsubsection*{Without resampling}

Let $\boldsymbol{\mathcal{T}}^{\, \bullet} = (\mathcal{T}, d, \mu, \rho, x)$ be a marked measured rooted tree and let  $g \colon [0, \zeta] \to \R$ be a continuous function with compact support such that 
\begin{equation}\label{eq:noreasmpling}
\mathrm{Ht}(x) - g(0) + \min_{[0, \zeta]} g \ge 0.
\end{equation}
Define the following modification:
\begin{equation}\label{eq:def_contour_modifie_continu}
h(t) = 2 g(t) - \min_{[t, \zeta]} g
,\end{equation}
for every $t \in [0, \zeta]$. 
Following~\eqref{eq:def_distance_arbre_fct}, let us then define the modified tree distance:
\begin{equation}\label{eq:def_distance_arbre_fct_modif}
D^g = d_h
\qquad\text{where }
h \text{ is given by~\eqref{eq:def_contour_modifie_continu}}
\text{ and }
d_h \text{ is given by~\eqref{eq:def_distance_arbre_fct}}
.\end{equation}
As for~\eqref{eq:def_distance_arbre_fct}, it is easy to see that $D^g$ is a pseudo-distance on $[0,\zeta]$ and its quotient is a measured real tree denoted by $\widetilde{\boldsymbol{\mathcal{T}}}_g^{\, \bullet}$, which should not be confused with the tree coded by $g$ that uses $d_g$. Note that the mass measure on $\widetilde{\boldsymbol{\mathcal{T}}}_g^{\, \bullet}$ has total mass $\zeta$. 
This tree comes with three natural distinguished points: the projection $\tilde{\rho}$ of the argmin of $g$ where the tree is rooted, the projection $\tilde{x}$ of $0$ (which will be identified with $x$), and that $\tilde{\zeta}$ of $\zeta$ (which will be the marked point of the new tree). 

Let $\delta = g(0) - \min g$ and for every $r \in [0, \delta]$, let then $\tilde{x}_r \in \widetilde{\boldsymbol{\mathcal{T}}}_g^{\, \bullet}$ be the projection of the time $\inf\{t \in [0,\zeta] \colon g(0)-g(t) = r\}$. Because of the modification~\eqref{eq:def_distance_arbre_fct_modif} the path $\llbracket \tilde{x}_0, \tilde{x}_\delta\rrbracket = \llbracket \tilde{x}, \tilde{\rho}\rrbracket$ describes the ancestral line of $\tilde{x}$, with $\tilde{x}_r$ lying at distance $2r$ from $\tilde{x}$ for every $r \in [0,\delta]$.

\begin{figure}[h!]
    \centering
    \includegraphics[page=7, width=.95\linewidth]{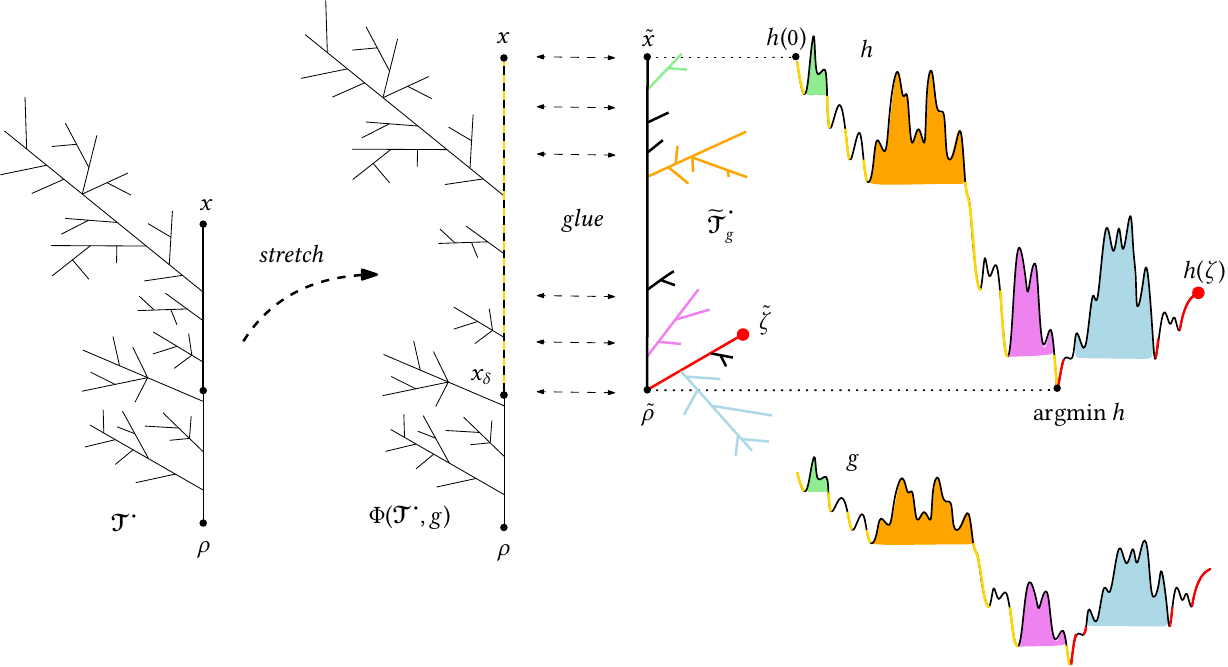}
    \caption{The construction of the continuum stretching and gluing mapping $\Phi$ from~\eqref{eq:def_Phi}.}
    \label{fig:def_Phi}
\end{figure}

Similarly in the tree $\mathcal{T}$, for every $r \in [0,\delta]$, let $x_r$ denote the point on the path from $x$ to the root $\rho$ that lies at distance $r$ from $x$. Its existence follows because Assumption~\eqref{eq:noreasmpling} states that $d(x,\rho) \ge \delta$.
Then let us double in $\mathcal{T}$ the distances along the path from $x$ to $x_\delta$ without any change anywhere else and then, for every $r \in [0, \delta]$, let us identify the point $x_r$ to the point $\tilde{x}_r \in \widetilde{\boldsymbol{\mathcal{T}}}_g^{\, \bullet}$.
We let finally
\begin{equation}\label{eq:def_Phi}
\Phi(\boldsymbol{\mathcal{T}}^{\, \bullet},g)
\end{equation}
be the metric space obtained after these identifications.
An easy exercise shows that this is still a tree. 
We equip $\Phi(\boldsymbol{\mathcal{T}}^{\, \bullet},g)$ with the natural measure, in which the measures on  $\boldsymbol{\mathcal{T}}^{\, \bullet}$ and $\widetilde{\boldsymbol{\mathcal{T}}}_g^{\, \bullet}$ are simply (pushed-forward by the stretching and) added together.
We shall keep the root $\rho$ of $\mathcal{T}$ as the root of $\Phi(\boldsymbol{\mathcal{T}}^{\, \bullet},g)$ and forget the point $x=\tilde{x}$ and instead mark the point $\tilde{\zeta}$. 

If $f_1, f_2$ are two continuous functions defined on possibly different intervals $[0,\zeta_1]$ and $[0,\zeta_2]$ respectively and with values in the same metric space $(E,d)$, then we extend the uniform distance as:
\[d_\infty(f_1, f_2) = |\zeta_1 - \zeta_2| + \sup_{s \in [0,1]} d(f_1(s \zeta_1), f_2(s \zeta_2)).\]
We shall use this notion in both cases $E=\R$ and $E = \mathbb T^{\, \bullet}$.

\begin{prop}[Continuity of $\Phi$]\label{prop:Phi_continue}
Let $\boldsymbol{\mathcal{T}}^{\, \bullet}_1= (\boldsymbol{\mathcal{T}}_1,x_1)$ and $\boldsymbol{\mathcal{T}}^{\, \bullet}_2= (\boldsymbol{\mathcal{T}}_2,x_2)$ be two marked measured compact real trees, and let $g_1 \colon [0, \zeta_1] \to \mathbb{R}$ and $g_2 \colon [0, \zeta_2] \to \mathbb{R}$ be two continuous functions with compact support satisfying $g_1(0)=g_2(0)=0$ and
\[\mathrm{Ht}(x_1)+ \min_{[0,\zeta_1]}g_1 \geq 0 
\qquad\text{and}\qquad
\mathrm{Ht}(x_2)+ \min_{[0,\zeta_2]}g_2 \geq 0.\]
Then there exists a universal constant $C>0$ such that
\[ d_{\mathrm{GHP}}^{\, \bullet}(\Phi(\boldsymbol{\mathcal{T}}_1^{\, \bullet}, g_1), \Phi(\boldsymbol{\mathcal{T}}_2^{\, \bullet}, g_2))
\le  C \bigl(d_{\mathrm{GHP}}^{\, \bullet}(\boldsymbol{\mathcal{T}}_1^{\, \bullet}, \boldsymbol{\mathcal{T}}_2^{\, \bullet}) + d_\infty(g_1, g_2)\bigr).\]
\end{prop}

\begin{proof}
In the previous subsection, we mentioned that the mapping $h \mapsto \mathcal{T}_h$ was Lipschitz-continuous~\cite[Proposition~2.10]{ADH13}.
So is the stretching by a factor $2$ along a branch applied to $\boldsymbol{\mathcal{T}}^{\, \bullet}$, and combining these two arguments, we obtain that the mapping $g \mapsto \widetilde{\boldsymbol{\mathcal{T}}}_g$ is Lipschitz-continuous as well.
Finally the gluing operation is Lipschitz-continuous.
\end{proof}

In addition to continuity, it is easy to check from the construction that $\Phi$ enjoys the following semi-group property:
let $g \oplus h$ be the concatenation of the functions $g$ and $h$, and assume that $g \oplus h$ is continuous and satisfies~\eqref{eq:noreasmpling}, then
\[
\Phi(\Phi(\boldsymbol{\mathcal{T}}^{\, \bullet},g), h) = \Phi(\boldsymbol{\mathcal{T}}^{\, \bullet}, g \oplus h)
.\]
This allows to define a process version of the previous construction.
Let us denote by $C([0,\zeta], \mathbb T^{\, \bullet})$ the space of continuous functions on $[0,\zeta]$ with values in the space of rooted and marked measured real trees, equipped as usual with the uniform distance $d_\infty$.
Given $\boldsymbol{\mathcal{T}}^{\, \bullet}$ and $g \colon [0,\zeta] \to \R$ satisfying~\eqref{eq:noreasmpling}, let us define
a function $\Psi(\boldsymbol{\mathcal{T}}^{\, \bullet},g) \in C([0,\zeta], \mathbb T^{\, \bullet})$ by setting for every $t \in [0,\zeta]$:
\[\Psi(\boldsymbol{\mathcal{T}}^{\, \bullet},g)_t = \Phi(\boldsymbol{\mathcal{T}}^{\, \bullet},g|_{[0,t]})
.\]
Continuity of $\Phi$ shows that this function is indeed continuous in $t$: simply apply Proposition~\ref{prop:Phi_continue} to the same trees $\boldsymbol{\mathcal{T}}^{\, \bullet}_1 = \boldsymbol{\mathcal{T}}^{\, \bullet}_2 = \boldsymbol{\mathcal{T}}^{\, \bullet}$ and to the functions $g|_{[0,t]}$ and $g|_{[0,t+s]}$ with $s>0$ small.

\subsubsection*{With resampling}

Let us now incorporate random resamplings in the previous construction.

We shall extend the definition of $\Psi$ without requiring Condition~\eqref{eq:noreasmpling}, but instead that the marked point $x \in \mathcal{T}$ be different from the root $\rho$.
Also, we now let $g$ be a continuous real-valued function, defined on the infinite half-line $[0, \infty)$ instead of a compact interval. The construction below is adapted in the obvious way if $g$ is only defined on a compact interval $[0,\zeta]$: simply stop at time $t=\zeta$.
Without loss of generality, let us also assume that $g(0)=0$.

Let us define by induction a \emph{random c\`adl\`ag} function $\Psi(\boldsymbol{\mathcal{T}}^{\, \bullet}, g) = (\boldsymbol{\mathcal{T}}^g_t)_{t \ge 0} = (\mathcal{T}_t,d_t, \mu_t, \rho_t,x_t)_{t \ge 0}$.
First let $\resamp_0 = 0$ and $\boldsymbol{\mathcal{T}}^g_0 = \boldsymbol{\mathcal{T}}^{\, \bullet}$.
Fix $i \ge 1$, suppose constructed $(\boldsymbol{\mathcal{T}}^g_t)_{t \le \resamp_{i-1}}$ and let us construct $(\boldsymbol{\mathcal{T}}^g_t)_{\resamp_{i-1} < t \le \resamp_i}$. Let us denote by $\delta_i$ the distance between the  root $\rho_{\resamp_{i-1}}$ and the marked point  $x_{\resamp_{i-1}}$ in $\boldsymbol{\mathcal{T}}^g_{\resamp_{i-1}}$, which we assume is nonzero.
Let then
\[\resamp_i = \inf\{t \ge 0 \colon g(t) < -(\delta_1 + \dots + \delta_i)\}.\] 
Let $g_i$ be the function defined by $g_i(s) = g(\resamp_{i-1}+s)-g(\resamp_{i-1})$ for every $s \in [0,\resamp_i-\resamp_{i-1}]$, so $g_i$ satisfies~\eqref{eq:noreasmpling} with the tree $\boldsymbol{\mathcal{T}}^g_{\resamp_{i-1}}$. 
We then define for  $s \in [0,\resamp_i-\resamp_{i-1})$:
\[\boldsymbol{\mathcal{T}}^g_{\resamp_{i-1}+s} = \Phi(\boldsymbol{\mathcal{T}}^g_{\resamp_{i-1}}, g_i|_{[0,s]}).\] 
By taking left-limit we can then define the tree $\boldsymbol{\mathcal{T}}^g_{\resamp_{i}-} \in \mathbb{T}^{\, \bullet}$ but  the marked point $x_{\resamp_i-}$ coincides with the root $\rho_{\resamp_i-}=\rho$, so the function $\Phi$ cannot apply anymore. We then perform a \textbf{resampling step} (this uses extra-randomness and destroys the continuity for the marked GHP topology): we define $\boldsymbol{\mathcal{T}}^g_{\resamp_{i}}$ as the same rooted tree as $\boldsymbol{\mathcal{T}}^g_{\resamp_{i}-}$, simply replacing the marked point $x_{\resamp_i-}=\rho_{\resamp_i-}=\rho$ by a random point $x_{\resamp_i}  \in \boldsymbol{\mathcal{T}}^g_{\resamp_{i}-}$  \textbf{different from the root}, sampled independently of the past proportionally to the measure $\mu_{\resamp_i-}(\cdot \cap (\mathcal{T}^g_{\resamp_i-} \setminus \{\rho\}))$. Observe that the latter is not the null measure, even if the original measure $\mu$ on $\boldsymbol{\mathcal{T}}$ was concentrated on the root $\rho$.
These times $\resamp_i$ are called the \emph{resampling instants}.

It could be that $\resamp_i=\infty$ for some $i \ge 1$, in which case the dynamics continues forever using the function $\Phi$ and without resampling the marked point anymore.
The next lemma rules out however the possibility that resampling instants accumulate, so the random dynamics is well-defined for every time $t \ge 0$.

\begin{lem}[$\Psi$ is well-defined]\label{lem:pas_d_accumulation}
With the above notation, almost surely either there exists $n \ge 1$ such that $\resamp_n=\infty$ or $\resamp_n \to \infty$ as $n \to \infty$. In particular $\Psi$ is a well-defined random continuous function of marked trees.
\end{lem}

\begin{proof}
Recall that we assume that in the initial tree $\boldsymbol{\mathcal{T}}$, the marked point $x$ is different from the root $\rho$. Then even if the initial measure $\mu$ is concentrated on $\rho$, for some small time $t>0$, the dynamics has not required yet any resampling and the measure $\mu_t$ on the current tree is not concentrated on the root, so for some $\varepsilon, \delta> 0$, the mass of points at distance greater than $\delta$ from the root is at least $\varepsilon$.

Then at each resampling instant $\resamp_i$ before time $K$, the probability to resample the new point $x_{\resamp_i}$ in the image of the previous set is bounded below by $\varepsilon/(\mu(\mathcal{T})+K) > 0$. When this occurs, since the dynamics only increases the distances, then $x_{\resamp_i}$ lies at distance greater than $\delta$ from the root in the current tree; since the function $g$ is uniformly continuous on $[0,K]$, then in this case, the next resampling instant cannot be arbitrarily close to the present one. Consequently we can only resample finitely many times before time $K$ a new marked point at height greater than $\delta$, and since this has a probability bounded away from $0$ to occur at each resampling step, then there can only be finitely many resampling steps at all before time $K$.
\end{proof}

This construction gives rise to a process of growing random trees $(\boldsymbol{\mathcal{T}}^g_t)_{t \ge 0}$ started from any tree with a marked point distinct from the root. Beware though, the law of this process is \emph{not} continuous with respect to the function $g$. This is due mainly to the behaviour of $g$ near the resampling instants, e.g.~if it makes a local minimum or a plateau where it stays constant.
Such pathological behaviours however do not occur when considering a random function such as Brownian motion $B$ as we will do below.

\subsection{Scaling limit of the discrete dynamics}

Let us finally prove that this continuum construction with an independent Brownian motion as driving function is the scaling limit of the discrete dynamics as claimed in Theorem~\ref{thm:cv_BBJ_resampling}.

Let us first consider the dynamics between two consecutive resampling steps.
The framework is the same as in Proposition~\ref{prop:cv_hauteur_entre_resampling}: for every $n \ge 1$, we let the discrete dynamics $(T^{(n)}_k, X^{(n)}_k)_{k \ge 0}$ start from a marked binary tree $(T^{(n)}_0, X^{(n)}_0) = (t^{(n)}, x^{(n)})$, not necessarily with $n$ internal vertices, and where $x^{(n)}$ has height $h_n$. Assume that $t^{(n)}$ is equipped with a finite measure $\mu^{(n)}$ on its vertices.
Let $\Resamp^{(n)}$ denote the first resampling time.
We shall equip the tree $T^{(n)}_k$ with the measure $\mu^{(n)}_k$ given by the image of $\mu^{(n)}$ on the vertices of $t^{(n)}$ and a mass $1$ on each new vertex.

\begin{thm}\label{thm:cv_BBJ_entre_deux_resampling}
Assume that the initial non-marked tree $t^{(n)}$ is $n$-good in the sense of Definition~\ref{def:good} and that, as a marked tree, it satisfies the convergence in the marked GHP topology on $\mathbb{T}^{\, \bullet}$:
\[\Bigl(\frac{1}{\sqrt{2n}} t^{(n)}, \frac{1}{2n} \mu^{(n)}, x^{(n)}\Bigr) \cv \boldsymbol{\mathcal{T}}^{\, \bullet} = (\mathcal{T}, d, \mu, \rho, x),\]
for some limit tree in which $\mathrm{Ht}(x)>0$.
Let $B=(B_t)_{t \ge 0}$ be an independent standard Brownian motion and $\resamp = \inf\{t>0 \colon B_t<-\mathrm{Ht}(x)\}$.
Then we have
\begin{equation}\label{eq:conv1}
\Bigl(\frac{\Resamp^{(n)}}{n}, \Bigl(\frac{1}{\sqrt{2n}} T^{(n)}_{\lfloor tn\rfloor \wedge (\Resamp^{(n)}-1)}, \frac{1}{2n} \mu^{(n)}_{\lfloor tn\rfloor \wedge (\Resamp^{(n)}-1)}, X^{(n)}_{\lfloor tn\rfloor \wedge (\Resamp^{(n)}-1)}\Bigr)_{0 \leq t < \lfloor \Resamp^{(n)}/n \rfloor}\Bigr)
\cvloi
(\resamp, (\Psi(\boldsymbol{\mathcal{T}}^{\, \bullet}, B)_{t})_{0 \leq t < \resamp})
,\end{equation}
for the Skorokhod topology.
\end{thm}

\begin{proof}
The strategy of the proof consists in replacing in $T^{(n)}_k$ the graph distance $d^{(n)}_k$ and the measure $\mu^{(n)}_k$ by the discrete analogue of the construction of $\Psi$, using the height process $S^{(n)}$ of the marked vertex, that we shall denote by $\tilde{d}^{(n)}_k$ and $\tilde{\mu}^{(n)}_k$ respectively, and proving that the GHP distance between $(T^{(n)}_k, \frac{1}{\sqrt{2n}} d^{(n)}_k, \frac{1}{2n} \mu^{(n)}_k, X^{(n)}_k)$ and $(T^{(n)}_k, \frac{1}{\sqrt{2n}} \tilde{d}^{(n)}_k, \frac{1}{2n} \tilde{\mu}^{(n)}_k, X^{(n)}_k)$ tends to $0$ in probability as $n \to \infty$ uniformly over $k \le \Resamp^{(n)}$.
We then prove our claim for this modified tree-process, using the convergence of $S^{(n)}$ and $\Resamp^{(n)}$ from Proposition~\ref{prop:cv_hauteur_entre_resampling}.

Fix $\varepsilon>0$. 
With probability at least $1-\varepsilon$, the first resampling time $\Resamp^{(n)}$ is bounded by some large constant $C$ times $n$ (thanks to~\eqref{eq:cv_hauteur_brownien}) and all marked trees $(T^{(n)}_k, X^{(n)}_k)$ for $k < \Resamp^{(n)}$ are $Cn$-good, thanks to Lemma~\ref{lem:all_good}.
We henceforth assume that these events do occur.

\textsc{Step 1: modifying the image of the initial tree.}
Let us first focus on the image of the vertices of the initial tree $t^{(n)}$ in the tree $T^{(n)}_k$.
Recall the description of $t^{(n)}$ as the ancestral lineage of $x^{(n)}$ along which are grafted subtrees that will not change at all until time $\Resamp^{(n)}$ (included).
Only the ancestral lineage of $x^{(n)}$ will be stretched by grafting new subtrees on it.
Precisely, recall that at each time of a new strict negative record by $S^{(n)}$, the discrete tree is marked at the next turning vertex along the ancestral line of $x^{(n)}$; since these vertices are $n$-good by our assumption that the initial tree is, then Lemma~\ref{lem:stretching_discret} applies to show that, up to a $o(\sqrt{n})$-error, their ancestral lineage will be stretched at time $k$ by the negative of the infimum of $S^{(n)}$ since their visit.
A union bound shows that the amount of stretching between $x^{(n)}$ and the first turning vertex, then between two turning vertices, is approximately, up to a $o(\sqrt{n})$ factor, given by the differences in the successive strict negative records of $S^{(n)}$, uniformly over time up to $\Resamp^{(n)}$.
Let $\tilde{d}^{(n)}_k$ denote the distance between the vertices of $t^{(n)}$ after this stretch using $(S^{(n)}_i)_{i \le k}$; then, 
\[\frac{1}{\sqrt{n}} \sup_{k \le \Resamp^{(n)}} \sup_{a,b \in t^{(n)}} |d^{(n)}_k(a, b) - \tilde{d}^{(n)}_k(a, b)| \cvproba 0.\]
As far as measures are concerned, let us define $\tilde{\mu}^{(n)}_k(a) = \mu^{(n)}_k(a) = \mu^{(n)}(a)$ for every $a \in t^{(n)}$, without any modification.

\textsc{Step 2: removing the new leaves.}
Let us turn to the subtrees that are created by the dynamics.
At time $k \le \Resamp^{(n)}$, they form a forest $\widetilde{T}\vphantom{T}^{(n)}_k$ that we equip with the graph distance $d^{(n)}_k$ inherited from $T^{(n)}_k$ (equivalently, we keep $x^{(n)}$ and its ancestral line) and the measure $\mu^{(n)}_k$ that gives a mass $1$ to every vertex.
As in the previous section, let us order these vertices from $2$ to $2k+1$ in order of appearance: the vertices $2i$ and $2i+1$ are respectively the internal one and the leaf created at step $i$. For definiteness, let $0$ and $1$ also denote the vertex $2$.

We shall focus on internal vertices only and define $\overline{2i} = \overline{2i+1} = 2i$ for every $1 \le i \le k$, i.e.~we project each new leaf on the internal vertex with preceding label.
Define then $\overline{\mu}^{(n)}_k$ as the image of $\mu^{(n)}_k$ by this projection and $\overline{d}^{(n)}_k(i, j) = d^{(n)}_k(\overline{i}, \overline{j})$ for every $i, j \le 2k+1$.
According to Lemma~\ref{lem:distance_interne_feuille} the distance between the leaf $2i+1$ and the internal vertex $2i$ can only increase over time and, at time $\Resamp^{(n)}$, it has the geometric distribution with parameter $1/2$. We infer from a union bound (recall the Hoeffing inequality in the proof of Lemma~\ref{lem:stretching_discret}) that, uniformly over all time up to $\Resamp^{(n)}$, which we recall is at most $Cn$ with high probability, and uniformly over all leaves up to this time, their distance to the internal vertex with the preceding label is $o(\sqrt{n})$. In other words,
\[\frac{1}{\sqrt{n}} \sup_{k \le \Resamp^{(n)}} \sup_{a,b \in \widetilde{T}\vphantom{T}^{(n)}_k} |d^{(n)}_k(a, b) - \overline{d}^{(n)}_k(a, b)| \cvproba 0.\]
\textsc{Step 3: modifying the new internal vertices.}
Let us modify now the distances $\overline{d}^{(n)}_k$ on $\widetilde{T}\vphantom{T}^{(n)}_k$ as in the first step. 

In the forest $\widetilde{T}\vphantom{T}^{(n)}_k$, the distance between the corresponding internal vertices is given by the sum of their height in the tree $T^{(n)}_k$ minus the height of their last common ancestor. Recall that the height of an internal vertex $2i$ at the time it is created is given by the path $S^{(n)}_{i-1}$; also Lemma~\ref{lem:ancetre_commun} allows to find the last common ancestor using this path, and finally Lemma~\ref{lem:stretching_discret} shows how the height evolves after it is created. For the latter fact, recall that with high probability, the marked trees up to time $\Resamp^{(n)} \le Cn$ are all $Cn$-good.
Just as in Step 1, define then the distance $\tilde{d}^{(n)}_k$ between two internal vertices $2i$ and $2j$ of $\widetilde{T}\vphantom{T}^{(n)}_k$ as
\begin{equation}\label{eq:modif_dist_arbre_tilde}
\tilde{d}^{(n)}_k(2i,2j) = \bigl(2 S^{(n)}_{i-1} - \min_{[i-1, k]} S^{(n)}\bigr) + \bigl(2 S^{(n)}_{j-1} - \min_{[j-1, k]} S^{(n)}\bigr) - 2 \min_{i-1 \le \ell \le j-1} \bigl(2 S^n_\ell - \min_{[\ell, k]} S^{(n)}\bigr)
.\end{equation}
As in Step 1, according to Lemma~\ref{lem:stretching_discret}, with high probability this distance is close to $\overline{d}^{(n)}_k$ up to $o(\sqrt{n})$ error, uniformly in $k \le \Resamp^{(n)}$.
As for the measures, as previously let $\tilde{\mu}^{(n)}_k = \overline{\mu}^{(n)}_k$, without further modification.

\textsc{Step 4: comparison with the continuous dynamics.}
By definition of the GHP, using as correspondence the identity on the image of the vertices of $t^{(n)}$ and the projection $i \mapsto \overline{i}$ for the new vertices, the first two steps have shown that the GHP distance between the original tree-process
\[\Bigl(T^{(n)}_{\lfloor tn\rfloor \wedge \Resamp^{(n)}}, \frac{1}{\sqrt{2n}} d^{(n)}_{\lfloor tn\rfloor \wedge \Resamp^{(n)}}, \frac{1}{2n} \mu^{(n)}_{\lfloor tn\rfloor \wedge \Resamp^{(n)}}, X^{(n)}_{\lfloor tn\rfloor \wedge \Resamp^{(n)}}\Bigr)_{t \geq 0}\]
and the modified one
\[\Bigl(T^{(n)}_{\lfloor tn\rfloor \wedge \Resamp^{(n)}}, \frac{1}{\sqrt{2n}} \tilde{d}^{(n)}_{\lfloor tn\rfloor \wedge \Resamp^{(n)}}, \frac{1}{2n} \tilde{\mu}^{(n)}_{\lfloor tn\rfloor \wedge \Resamp^{(n)}}, X^{(n)}_{\lfloor tn\rfloor \wedge \Resamp^{(n)}}\Bigr)_{t \geq 0}\]
tends to $0$ in probability.
Let us finally argue that the latter converges to our desired limit $(\Psi(\boldsymbol{\mathcal{T}}^{\, \bullet}, B)_{t \wedge \resamp})_{t \geq 0}$.

Let us fix some time $t > 0$ and assume $t < \resamp$; recall that $\Psi(\boldsymbol{\mathcal{T}}^{\, \bullet}, B)_t$ is constructed from $\boldsymbol{\mathcal{T}}^{\, \bullet}$ by stretching the distances by a factor $2$ along the ancestral line of the marked point and gluing this path to a new tree $\widetilde{\boldsymbol{\mathcal{T}}}_t^{\, \bullet}$ coded by the pseudo-distance $D^g$ from~\eqref{eq:def_distance_arbre_fct_modif} where $g=(B_s)_{s \le t}$. Precisely $\widetilde{\boldsymbol{\mathcal{T}}}_t^{\, \bullet}$ is the quotient of the interval $[0,t]$ by the equivalence relation $D^g=0$. This provides a natural correspondence between $\widetilde{\boldsymbol{\mathcal{T}}}_t^{\, \bullet}$ and $\widetilde{T}\vphantom{T}^{(n)}_{\lfloor nt\rfloor}$ by letting correspond the image of a time $s \in [0,t]$ in $\widetilde{\boldsymbol{\mathcal{T}}}_t^{\, \bullet}$ and the vertex labelled $\lfloor ns \rfloor$ in $\widetilde{T}\vphantom{T}^{(n)}_{\lfloor nt\rfloor}$.
Comparing the formula~\eqref{eq:modif_dist_arbre_tilde} for the discrete distances and~\eqref{eq:def_distance_arbre_fct_modif} for the continuum one, and recalling that $S^{(n)}$ and $\Resamp^{(n)}$ converge after scaling to $B$ and $\resamp$ by Proposition~\ref{prop:cv_hauteur_entre_resampling}, we infer that, jointly, the processes $(\frac{1}{\sqrt{2n}} \widetilde{T}\vphantom{T}^{(n)}_{\lfloor nt\rfloor \wedge \Resamp^{(n)}})_{t \ge 0}$ converge in distribution to $(\widetilde{\boldsymbol{\mathcal{T}}}_{t \wedge \resamp}^{\, \bullet})_{t \ge 0}$ in the GHP topology.
Similarly, since $t^{(n)}$ converges after scaling in the GHP topology towards $\boldsymbol{\mathcal{T}}$, using the image of the correspondence between these trees at time $t$, we conclude by comparing the discrete distance $\tilde{d}^{(n)}_{\lfloor tn\rfloor \wedge \Resamp^{(n)}}$ and the continuum one that the trees at time $t$ converge.
\end{proof}

Theorem~\ref{thm:cv_BBJ_entre_deux_resampling} does not readily extend to the first resampling time. More precisely,~\eqref{eq:conv1} shows that, just before resampling, the marked tree $ (\frac{1}{\sqrt{2n}} T^{(n)}_{\Resamp^{(n)}-1}, \frac{1}{2n} \mu^{(n)}_{\Resamp^{(n)}-1}, X^{(n)}_{\Resamp^{(n)}-1})$ converges towards $\Psi( \boldsymbol{ \mathcal{T}}^{\, \bullet},B)_{\resamp -}$ in the marked GHP sense. In general this \emph{does not imply}  that   
\begin{equation}\label{eq:convresamp}
\Bigl(\frac{1}{\sqrt{2n}} T^{(n)}_{\Resamp^{(n)}}, \frac{1}{2n} \mu^{(n)}_{\Resamp^{(n)}}, X^{(n)}_{\Resamp^{(n)}}\Bigr)  \cvloi \Psi( \boldsymbol{ \mathcal{T}}^{\, \bullet},B)_{\resamp}.
\end{equation}
The issue is that if the measure of $ \boldsymbol{ \mathcal{T}}^{\, \bullet}$ possesses an atom at the root, then, in the discrete setting we may resample $X^{(n)}_{\Resamp^{(n)}}$ close to the root, while the marked point in  $\Psi( \boldsymbol{ \mathcal{T}}^{\, \bullet},B)_{\resamp}$ is sampled according to the measure of $\Psi( \boldsymbol{ \mathcal{T}}^{\, \bullet},B)_{\resamp-}$ \emph{except} its possible atom at the root. However, as soon as $ \boldsymbol{ \mathcal{T}}^{\, \bullet}$ has no atom at the root, then~\eqref{eq:convresamp} is granted. By the Markovian nature of the discrete and continuous process, one may then extend the convergence up to finitely many resampling times: 
the convergence~\eqref{eq:convresamp} serves as a new initial condition, and by Lemma~\ref{lem:all_good}, the non-marked tree $T^{(n)}_{\Resamp^{(n)}}$ is $\max(n,\Resamp^{(n)})$-good and $\Resamp^{(n)}$ is of order $n$ by~\eqref{eq:conv1}, leading to the convergence~\eqref{eq:conv1} up to the second resampling time, etc.
Finally, recall from Lemma~\ref{lem:pas_d_accumulation} that resampling instants do not accumulate, then the convergence~\eqref{eq:conv1} extends to convergence over any compact intervals, i.e.~in the Skorokhod topology on $\D( \R_{+}, \mathbb{T}^{\, \bullet})$.

Theorem~\ref{thm:cv_BBJ_resampling} in the Introduction claims that such a convergence also holds when the discrete dynamics starts from a single vertex, which does not fit the assumption of Theorem~\ref{thm:cv_BBJ_entre_deux_resampling} but it can easily be deduced from this.

\begin{proof}[Proof of Theorem~\ref{thm:cv_BBJ_resampling}]
Fix $\delta>0$. The discrete dynamics started from a single vertex produces at time $\lfloor \delta n \rfloor$ a uniform random binary tree with size $\lfloor \delta n \rfloor$ and equipped with an independent uniform random vertex. By Aldous' invariance principle~\eqref{eq:aldous} (extended to take care of the measure), we then have
\[\frac{1}{\sqrt{2n}} \cdot \boldsymbol T_{\lfloor \delta n\rfloor}^{\, \bullet} \cvloi \boldsymbol{\mathcal{T}}^{\, \bullet}_\delta,\]
where the limit is a Brownian CRT with mass $\delta$ marked at an independent random point sampled from its mass measure.
We may then apply Theorem~\ref{thm:cv_BBJ_entre_deux_resampling} and the above discussion to derive the convergence of processes:
\[\Bigl(\frac{1}{\sqrt{2n}} \cdot \boldsymbol T^{\, \bullet}_{\lfloor t n\rfloor}\Bigr)_{t \ge \delta} \cvloi (\boldsymbol{\mathcal{T}}^{\, \bullet}_t)_{t \ge \delta} = \Psi(\boldsymbol{\mathcal{T}}^{\, \bullet}_\delta, B|_{[\delta, \infty)}).\]
To finish, let us remark that, as $\delta \to 0$, the sequence $(\frac{1}{\sqrt{2n}} \cdot \boldsymbol T^{\, \bullet}_{\lfloor t n\rfloor})_{0 \le t \le \delta}$ converges in probability towards a singleton uniformly in $n$. These facts imply the convergence stated in Theorem~\ref{thm:cv_BBJ_resampling}. 
\end{proof}

\section{R\'emy's diffusion is ergodic}
\label{sec:ergo}

In this section we prove Theorem~\ref{thm:ergo} on unique ergodicity of R\'emy's continuous dynamics.
Henceforth, we assume that the initial tree $\boldsymbol{\mathcal{T}}^{\, \bullet}_0$ has mass $1$ and is marked at a point different from its root. We let $B$ be an independent Brownian motion.
In order to lighten the notation, we shall set for every $t \ge 0$:
\[\boldsymbol{\mathcal{T}}_t^{\, \bullet} = \Psi( \boldsymbol{\mathcal{T}}^{\, \bullet}_0,B)_t
\qquad\text{and}\qquad
\boldsymbol{\mathcal{S}}^{\, \bullet}_t = \frac{1}{\sqrt{1+t}} \cdot \boldsymbol{\mathcal{T}}_t^{\, \bullet}.\]
In this way $\boldsymbol{\mathcal{S}}^{\, \bullet}_t$ has mass $1$ and its marked point differs from its root for every $t \ge 0$.
Let $\boldsymbol{\mathcal{T}}_{2\mathbb{e}}^{\, \bullet}$ denote the CRT coded by twice the standard Brownian excursion in the usual sense, together with a marked point sampled independently from its mass measure.
Throughout this section, by a \emph{marked Brownian CRT}, we mean a possibly scaled version $\boldsymbol{\mathcal{T}}_{2\mathbb{e}}^{\, \bullet}$, in which the marked point is always sampled independently from the mass measure.
Then Theorem~\ref{thm:ergo} claims that for any (possibly random) such initial tree $\boldsymbol{\mathcal{T}}^{\, \bullet}_0$, we have the convergence in distribution
\[\boldsymbol{\mathcal{S}}^{\, \bullet}_t
\cvloi[t]
\boldsymbol{\mathcal{T}}_{2\mathbb{e}}^{\, \bullet}
,\]
in the marked GHP topology.

To prove this result, we first observe that, from the discrete approximation, if $\boldsymbol{\mathcal{T}}^{\, \bullet}_0$ has the law of $\boldsymbol{\mathcal{T}}_{2\mathbb{e}}^{\, \bullet}$, then so does $\boldsymbol{\mathcal{S}}^{\, \bullet}_t$ for any $t>0$.
The convergence for a generic initial distribution then follows if we can prove that $\boldsymbol{\mathcal{S}}^{\, \bullet}_t$ stays eventually arbitrarily close to the analogue started from $\boldsymbol{\mathcal{T}}_{2\mathbb{e}}^{\, \bullet}$ when using as much as possible the same randomness from Brownian motion and resampling for both processes.
We stress however that the dynamics never entirely forgets its past and so we only provide an asymptotic coupling in the spirit of~\cite{HMS11}.
Although we could try to fit the very general framework of this reference, we have opted for a (rather short and) completely self-contained argument.

Our strategy consists in showing that, given any initial tree and $\varepsilon>0$, with high probability, there exists a (random) finite time $t$ at which the tree $\boldsymbol{\mathcal{T}}^{\, \bullet}_t$ contains a subtree whose mass represents a fraction greater than $1-\varepsilon$ of the total tree and that has the law of a scaled version of $\boldsymbol{\mathcal{T}}_{2\mathbb{e}}^{\, \bullet}$ (and whose mark is the mark of $\boldsymbol{\mathcal{T}}^{\, \bullet}_t$), and such that in addition the subsequent evolution of this subtree and the entire tree stay close to each other. Precisely, we shall prove that the GHP distance between the rescaled entire tree and the rescaled Brownian subtree tends to $0$ on this event with large probability, which allows to conclude. 

\subsection{A large Brownian subtree}

Let us first make formal the previous idea that the entire tree will contain at some finite time a large Brownian subtree that will occupy most of it.
Let us first introduce some denomination that we will use throughout this section.

\begin{defi}\label{def:good_ergo}
We say that a random marked tree $\boldsymbol{\mathcal{T}}^{\, \bullet}$ with mass $1$ is \emph{$\varepsilon$-good} where $\varepsilon \in (0, 1)$, if it is composed of three parts, as illustrated in Figure~\ref{fig:good_Brownian_tree}:
\begin{enumerate}
\item a \emph{spine}, in red, which is a segment started from the root and with length smaller than $\varepsilon$;
\item a \emph{good tree}, in green, grafted at the extremity of the spine, with mass greater than $1-\varepsilon$, which has the law of a scaled marked Brownian CRT, whose marked point is that of the entire tree;
\item a \emph{bad} part made of an arbitrary forest of blue trees grafted all along the spine.
\end{enumerate}
\end{defi}

\begin{figure}[h!]\centering
\includegraphics[page=9, height=10\baselineskip]{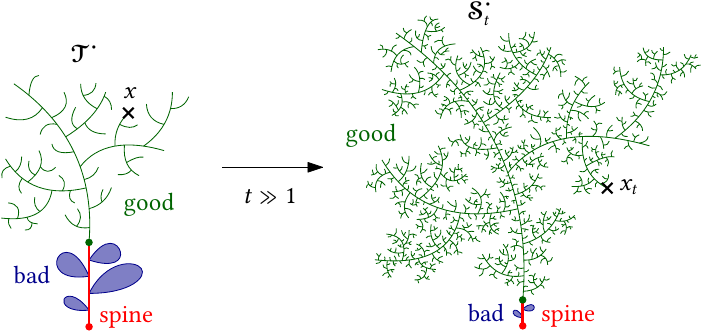}
\caption{Strategy of the proof of Theorem~\ref{thm:ergo}: If we start from the situation on the left with a big green Brownian subtree sitting at the end of a short red spine, then with high probability this Brownian subtree will take over the entire tree when applying R\'emy's growing diffusion. A key input is that with high probability we always resample the marked point in the green Brownian part.}
\label{fig:good_Brownian_tree}
\end{figure}

If at some time $t \ge 0$ in R\'emy's diffusion, the rescaled tree $\boldsymbol{\mathcal{S}}^{\, \bullet}_t$ is $\varepsilon$-good, then we shall still call \emph{good} and \emph{bad} the parts of the subsequent trees that will come from stretching or grafting on these initial good and bad parts respectively, and call \emph{spine} the parts coming from stretching the initial red spine, but new subtrees grafted on the spine will belong to the bad part.

\begin{lem}[Birth of a Brownian subtree]\label{lem:ergo_good}
Fix any marked measured tree $\boldsymbol{\mathcal{T}}^{\, \bullet}_0$ with mass $1$ and a marked point different from the root and fix some $\varepsilon \in (0, 1)$. 
Then there exists an almost surely finite stopping time $N \ge 1$ (for the discrete-time filtration generated by the dynamics at the resampling instants) such that $\boldsymbol{\mathcal{S}}^{\, \bullet}_{\resamp_{N}}$  is $\varepsilon$-good.
\end{lem}

\begin{proof}
Suppose first that the marked point in the initial tree lies at height $h \ge 1$ and let us prove the conclusion occurs with positive probability for $N=1$. Fix $K>0$ arbitrarily large and observe that the following event (depicted in Figure~\ref{fig:good_Brownian_motion} left)
about a Brownian motion $B$ started from $B_0=1$ occurs with positive probability:
first $\tau = \inf\{t>0 \colon B_t=0\} \in (K+1, K+2)$
and second, there exists 
a subexcursion above the running minimum, on the time interval say $[g,d]$, such that $g\leq 1, d \ge K+1$ and $B_g=B_d \le 1$.
Let us scale this picture by a factor $h$ in space and $h^2$ in time, then $h^2 \tau = \resamp_1$ is the first resampling time.
Conditionally on the previous event, the sub-excursion on the interval $(g,d)$ codes a Brownian subtree that occupies a fraction of the mass of the tree $\Psi(\boldsymbol{\mathcal{T}}^{\, \bullet}_0, B)_{\resamp_1}$ equal to, since $h^2 \ge 1$,
\[\frac{h^2 d - h^2 g}{1+h^2 \tau} \ge \frac{d-g}{1+\tau} \ge \frac{K}{K+3}.\]
Since $1+h^2 \tau \ge h^2 \tau \ge h^2 (K+1)$, then the root of this Brownian subtree lies at height at most
\[h 
\le \frac{1}{\sqrt{K+1}} \sqrt{1+h^2 \tau}
= \frac{1}{\sqrt{K+1}} \sqrt{1+\resamp_1}
.\]
We conclude that given any $K>0$, there exists $c_K>0$, depending only on $K$, such that, whatever the initial tree $\boldsymbol{\mathcal{T}}^{\, \bullet}_0$, provided that the marked point lies at height no smaller than $1$, at the first resampling instant $\resamp_1$, with probability at least $c_K$, the rescaled tree $\boldsymbol{\mathcal{S}}^{\, \bullet}_{\resamp_1}$ contains a Brownian subtree that occupies a fraction of the mass at least $1-3/(K+3)$ and whose root lies at height at most $1/\sqrt{K+1}$; finally the new marked point is sampled in the Brownian subtree from its mass measure.

Consider now the case of an initial marked point lying at height smaller than $1$ and recall that the initial tree $\boldsymbol{\mathcal{T}}^{\, \bullet}_0$ has mass $1$. Now for every $k \ge 1$, a Brownian motion started from $B_0=0$ has the same positive probability, independent of $k$, to satisfy $B_{2^k} > 2 \times \sqrt{2^k-2^{k-1}} + \inf_{[2^{k-1},2^k]} B$; see Figure~\ref{fig:good_Brownian_motion} right for a representation of this event for $k=1$. Hence almost surely this occurs for some finite index $k_0 \ge 1$ and then, whatever the initial tree $\boldsymbol{\mathcal{T}}^{\, \bullet}_0$, the marked point in the tree $\boldsymbol{\mathcal{S}}^{\, \bullet}_{2^{k_0}}$ lies at height greater than 
$2 \times \sqrt{2^{k_0}-2^{k_0-1}} / \sqrt{1+2^{k_0}} 
\ge 2 / \sqrt{3} > 1$.
According to the first part of this proof, there is a probability at least $c_K>0$ that at the first resampling time after $2^{k_0}$, the rescaled tree contains a Brownian subtree that occupies a fraction of the mass at least $1-3/(K+3)$ and whose root lies at height at most $1/\sqrt{K+1}$.
If this fails, then either the marked point at this resampling time lies at height greater than $1$ in the rescaled tree and we have the same probability $c_K$ that this succeeds at the next resampling time. Else, if the marked point at the resampling time lies at height smaller than $1$, then we wait for a finite time until we are sure to reach a height greater than $1$ before applying the first part of the argument again. In any case, after an almost surely finite time, the same conclusion as the first paragraph holds.
\end{proof}

\begin{figure}[h!]\centering
\includegraphics[page=10, width=.9\linewidth]{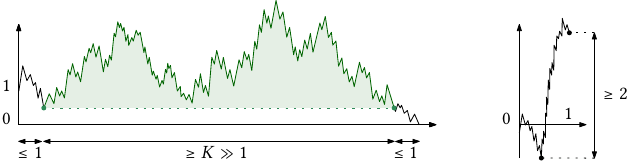}
\caption{How to get a good tree from an arbitrary tree: on the left a Brownian motion started from $1$ makes an enormous excursion above its minimum at a low height before reaching $0$; on the right a Brownian motion started from $0$ ends at time $1$ at least $2$ unit above its minimum.}
\label{fig:good_Brownian_motion}
\end{figure}

\subsubsection*{Rate of growth of Brownian trees}

The next key idea is to understand the rate of growth of the discrete-time chain made of the random trees in-between resampling times when the initial tree is a Brownian CRT.
This will eventually show that the Brownian subtree obtained from the previous lemma will grow more rapidly than the rest.
Precisely, let $(\resamp_n)_{n \ge 1}$ denote the resampling instants in the process $\Psi( \boldsymbol{\mathcal{T}}^{\, \bullet}_0,B)$ and also $\resamp_0=0$.

It follows from Theorem~\ref{thm:cv_BBJ_entre_deux_resampling} and Lemma~\ref{lem:indep_tree_failure} that if $\boldsymbol{\mathcal{T}}^{\, \bullet}_0 = \boldsymbol{\mathcal{T}}^{\, \bullet}_{2 \mathbb{e}}$ is a standard marked Brownian CRT, then conditionally on $(\resamp_n)_{n \geq 1}$, the renormalised trees $(\boldsymbol{\mathcal{S}}^{\, \bullet}_{\resamp_n})_{n \ge 0}$
form a stationary sequence.
In particular, the ratios of their masses $(1+\resamp_{n+1})/(1+\resamp_n)$ for $n \geq 0$ also form a stationary sequence.
We first compute this stationary law.

For every $a >0$, let $\tau_a = \inf \{t > 0 \colon B_t = -a\}$ denote the first hitting time of $-a$ by the standard Brownian motion,
then the first resampling time $\resamp_1$ is distributed as $\tau_{\mathrm{Ht}(X_0)}$ where  $\mathrm{Ht}(X_0)$ is the height of the random point $X_0$ in $\boldsymbol{\mathcal{T}}^{\, \bullet}_0$, which is independent of $B$.

\begin{lem}\label{lem:temps_atteinte_Rayleigh}
When $\boldsymbol{\mathcal{T}}^{\, \bullet}_0 = \boldsymbol{\mathcal{T}}^{\, \bullet}_{2 \mathbb{e}}$, the random variable $\log(1+\tau_{\mathrm{Ht}(X_0)})$ is exponentially distributed with expectation $2$. 
\end{lem}

\begin{proof}
For standard Brownian motion, $\tau_a$ has density $\frac{a}{\sqrt{2\pi t^3}}\e^{-a^2/(2t)}$ on $[0,\infty)$.
Then for twice the standard Brownian excursion, it is well-known that the height of a uniform random point follows the Rayleigh distribution with density $x \e^{-x^2/2}$ on $[0,\infty)$.
This can be obtained e.g.~from the Bismut decomposition or by taking the scaling limit of the height of a uniform random vertex in a uniform random binary tree.

A straightforward integration, using
the change of variables $(u, y) = (\log(1+t), x \sqrt{(1+t) / t})$,
whose Jacobian determinant equals $(t (1+t))^{-1/2}$,
shows that for any measurable and nonnegative function $f$, we have
\begin{align*}
\E[f(\log(1+\tau_{\mathrm{Ht}(X)}))]
&= \int_0^\infty \int_0^\infty f(\log(1+t)) \frac{x}{\sqrt{2\pi t^3}}\e^{-x^2/(2t)} x \e^{-x^2/2} \d x \d t
\\
&= \int_0^\infty \int_0^\infty f(u) \e^{-u/2} \frac{y^2}{\sqrt{2\pi}}\e^{-y^2/2} \d y \d u
\\
&= \int_0^\infty \frac{f(u)}{2} \e^{-u/2} \d u
.\end{align*}
This proves the claim.
\end{proof}

In the next proposition, we let $\boldsymbol{\mathcal{G}}_t$ denote the good part of $\boldsymbol{\mathcal{T}}_t$ and we denote by $\gamma_t$ its mass and by $\beta_t = (t+1)-\gamma_t$ the combined mass of the bad part and the spine.
If the marked point $x_t$ of $\boldsymbol{\mathcal{T}}^{\, \bullet}_t$ belongs to $\boldsymbol{\mathcal{G}}_t$, then the latter is marked at $x_t$; otherwise we mark $\boldsymbol{\mathcal{G}}_t$ at its root.

\begin{prop}[The Brownian subtree may take over]\label{prop:toy_ergo}
For any $\varepsilon>0$, there exists $c_\varepsilon>0$ such that if the initial tree $\boldsymbol{\mathcal{T}}^{\, \bullet}_0$ is $\varepsilon$-good, then with (annealed) probability at least $c_\varepsilon$, the following two events hold:
\begin{enumerate}
\item For every $n\ge 1$, the point that is resampled at time $\resamp_n$ in the tree $\boldsymbol{\mathcal{T}}_{\resamp_n}^{\, \bullet}$ always belongs to the good part $\boldsymbol{\mathcal{G}}_{\resamp_n}$,

\item and the proportion of the good mass $\gamma_t / (1+t)$ tends to $1$ and the marked GHP distance between the rescaled entire tree $\frac{1}{\sqrt{1+t}} \cdot \boldsymbol{\mathcal{T}}_t^{\, \bullet}$ and the good part $\frac{1}{\sqrt{\gamma_t}} \cdot \boldsymbol{\mathcal{G}}_t^{\, \bullet}$ tends to $0$. 
\end{enumerate}
\end{prop}

\begin{proof}
Recall that if we start the dynamics with a marked Brownian CRT of mass $1$, then the ratios $((1+\resamp_{n})/(1+\resamp_{n-1}))_{n \ge 1}$ of the masses between two resampling times form a stationary sequence.
Without assuming ergodicity of the rescaled trees $\boldsymbol{\mathcal{S}}^{\, \bullet}_{\resamp_n}$ (which is what we are eventually aiming to prove!), Birkhoff's ergodic theorem implies the convergence 
\begin{equation} \label{eq:birkoff}
\frac{1}{n}\sum_{k=1}^n \log \frac{1+\resamp_{k}}{1+\resamp_{k-1}} \cvps \Theta, 
\end{equation}
where, by uniform integrability, the limit $\Theta$ is a random variable that has expectation $\E[\log(1+\resamp_1)] = 2$ according to Lemma~\ref{lem:temps_atteinte_Rayleigh}. Notice that this expectation is relative to both the initial tree and the independent Brownian motion $B$ driving  R\'emy's diffusion (and the randomness needed for the resamplings).

Now suppose that the initial tree contains a good part, which is a marked Brownian CRT with mass $m$, whose marked point is that of the entire tree, sitting on top of a spine with length $h$. 
If the dynamics $\Psi$ only resamples points in the good tree, then in-between resampling events, it first stretches and glues in this good subtree, before reaching the spine, stretching it and gluing new bad subtrees up to the next resampling event.
By the strong Markov property, we may split the Brownian motion $B$ in two: the part $B^1$ affecting the good tree and the part $B^2$ affecting the spine and bad trees and both are independent Brownian motions.
Since the good part is initially a Brownian CRT, then after scaling by a factor $m$, its mass evolution satisfies~\eqref{eq:birkoff} with the $\resamp^1_k$'s corresponding to $B^1$ and with a limit $\Theta^1$ that satisfies $\E[\Theta^1] = 2$, so in particular $\P(\Theta^1 \ge 2)>0$.
As for the rest, at each step the length of the spine is simply doubled, and we add a forest coded by the Brownian motion $B^2$ started from the current length of the spine and killed upon reaching $0$.

Then after $n$ resampling steps, if all in the good part, this good part has mass
\[\gamma_{\resamp_n} = m (1+\resamp^1_n) 
= m \prod_{k=1}^n \frac{1+\resamp^1_k}{1+\resamp^1_{k-1}}
= m \e^{n (\Theta^1+o(1))}
,\]
where $o(1)$ is a term that tends almost surely to $0$,
while the spine and bad part have a combined mass
\[\beta_{\resamp_n} = (1-m) + \tau^2_{h(1+2+4+ \ldots + 2^{n-1})} ~\mathop{=}^{(d)}~ (1-m) + h^2 \tau^2_{2^n-1},\]
where $\tau^2_a$ is the hitting time of $-a$ by the Brownian motion $B^2$.

Suppose that the initial tree is $\varepsilon$-good, namely that $h<\varepsilon$ and $m>1-\varepsilon$.
Let us fix $\eta>0$ small enough so $\e^2 > 4+4\eta$. Easy estimates show that, with a positive probability, depending on $\varepsilon$, we have $\beta_{\resamp_n} \leq (4+ \eta)^{n}$ for all $n \ge 0$ large enough. Recall that $\P(\Theta \ge 2)>0$; since $\e^2 > 4+4\eta$, then again, with a positive probability, depending on $\varepsilon$, we have $\gamma_{\resamp_n} \geq (4+ 2\eta)^{n}$ for all $n \ge 0$ large enough. We infer that, a posteriori, the conditional probability to always resample the points in the good part equals
\[\prod_{n \ge 1} \frac{\gamma_{\resamp_n}}{\beta_{\resamp_n}+\gamma_{\resamp_n}},\]
which is positive with positive probability.
This proves the first claim.

For the second claim, note that on the previous event, after $n$ steps and without rescaling, the good tree is a Brownian CRT with mass $\gamma_{\resamp_n} \geq (4+ 2\eta)^{n}$, whereas the bad part is made of a spine with length $h(2^n-1)$ on which are grafted the initial bad trees that have not changed as well as a forest of Brownian CRT's with total mass at most $\beta_{\resamp_n} \le (4+\eta)^n$.
In addition, with a positive probability, depending again only on $\varepsilon$, this bad forest of Brownian CRT's only contains trees with diameter smaller than $(4+3\eta/2)^{n/2}$.
If we rescale the entire tree by $1/(\gamma_{\resamp_n}+\beta_{\resamp_n}) = 1/(\resamp_n+1)$, then the spine plus bad part only thus represent a tiny fraction in terms of both the mass and distances so the good part is eventually close to the entire tree for large enough resampling instants.

In order to conclude about any time $t$ large enough, we compare the situation at such a time and at the two surrounding resampling times using monotonicity.
Indeed, if the point that is resampled at time $\resamp_n$ belongs to the good part, then the dynamics will first affect this part only, so the rescaled good part will be even closer to the rescaled entire tree, until we leave the good part and reach the spine: then the dynamics will grow the bad part so the good part and the entire tree start to differ more and more, but only up to time $\resamp_{n+1}$ where there are close again.
\end{proof}

Let us prove now that the conclusion of Proposition~\ref{prop:toy_ergo} in fact holds with high probability when $\varepsilon$ is small. To this end, let us first improve the ergodic theorem~\eqref{eq:birkoff} by showing that the limit $\Theta$ is maximal. This relies on  Proposition~\ref{prop:toy_ergo} (that used~\eqref{eq:birkoff} in the first place!).

\begin{cor}[Ergodic theorem for the growth of Brownian trees]\label{cor:ergo_masses_resampling}
When $\boldsymbol{\mathcal{T}}_0^{\, \bullet} = \boldsymbol{\mathcal{T}}^{\, \bullet}_{2 \mathbb{e}}$, we have the almost sure convergence:
\[\frac{1}{n}\sum_{k=1}^n \log \frac{1+\resamp_{k}}{1+\resamp_{k-1}} \cvps \E[\log(1+\tau_{\mathrm{Ht}(X_0)})] = 2.\]
\end{cor}

\begin{proof}
According to~\eqref{eq:birkoff}, the almost sure convergence holds with a limit $\Theta$ that has expectation $\E[\log(1+\tau_{\mathrm{Ht}(X_0)})] = 2$. 
It suffices then to prove that $\Theta \ge 2$ almost surely.

Let us fix $\eta>0$ such that $\e^{2-\eta} > 4$ and such that $\P( \Theta = 2-\eta)=0$. Let $A = \{\Theta \le 2-\eta\}$, and let us suppose for contradiction that $\P(A) > 0$. Let us denote by $\Theta_n = \frac{1}{n}\sum_{k=1}^n \log \frac{1+\resamp_{k}}{1+\resamp_{k-1}}$ the left side of~\eqref{eq:birkoff} and let $A_n = \{\Theta_n \le 2 - \eta/2\}$.
By the convergence $\Theta_n \to \Theta$, the probability of the symmetric difference $A \Delta A_n$ tends to $0$.
Fix $\delta>0$ smaller than $\P(A)$ and let $n$ be large enough so $\P(A \Delta A_n) < \delta$. In particular $\P(A_n)>\P(A)-\delta>0$.
Note that the event $A_n$ belongs to 
$\mathscr{F}_n = \sigma(\boldsymbol{\mathcal{T}}_0^{\, \bullet}, (B_t)_{t \le \resamp_n}, (x_{\resamp_k})_{k \le n})$ the filtration generated by R\'emy's diffusion up to the $n$'th resampling time.

Now let us fix $\varepsilon>0$.
Applying Lemma~\ref{lem:ergo_good} and Proposition~\ref{prop:toy_ergo} to the tree $\boldsymbol{\mathcal{T}}^{\, \bullet}_{\resamp_n}$ shows that with probability $c_{\varepsilon, \eta}>0$, the following holds conditionally on $\mathscr{F}_n$:
there exists a time $\resamp_N$ for some finite stopping time $N \ge n$ at which the tree is $\varepsilon$-good and the $\widetilde{\Theta}$ corresponding to the good Brownian subtree satisfies $\{\widetilde{\Theta}>2-\eta\}$, and after that all resampled points belong to the good part. 
On this event, the bad part grows at a lower rate, of order $2^{2n}$, so the average growth rate
$\Theta_n$ of the entire tree converges almost surely to $\widetilde{\Theta} > 2-\eta$.
This shows that $\P(A_n \cap A^c) \ge c_{\varepsilon,\eta} \P(A_n) \ge c_{\varepsilon,\eta} (\P(A)-\delta)$.
Since this probability is also upper bounded by $\P(A \Delta A_n) < \delta$, then we must have
$c_{\varepsilon,\eta} (\P(A)-\delta) < \delta$.
Since $\delta$ can be chosen arbitrarily small without affecting $c_{\varepsilon,\eta}$ and $\P(A)$, this cannot always be true, unless $\P(A)=0$.

We have thus shown that $\P(\Theta \le 2-\eta) = 0$ for each $\eta>0$ such that $\e^{2-\eta} > 4$ and such that $\P( \Theta = 2-\eta)=0$. We may choose an decreasing sequence of such $\eta$'s that converge to $0$ to conclude that $\P(\Theta < 2) = 0$.
\end{proof}

Bouncing back one more time, we may now improve Proposition~\ref{prop:toy_ergo} using this improved version of~\eqref{eq:birkoff} by showing that the event in that proposition occurs with high probability when $\varepsilon$ is small enough.

\begin{cor}[The Brownian subtree does take over]\label{cor:toy_ergo_gde_proba}
The lower bound $c_\varepsilon > 0$ from Proposition~\ref{prop:toy_ergo} tends to $1$ as $\varepsilon \to 0$.
\end{cor}

\begin{proof}
Examining the proof of Proposition~\ref{prop:toy_ergo} shows that, now that we know that $\Theta=2$, the key point is to only resample points in the good part, which occurs with conditional probability
\[\prod_{n \ge 1} \frac{\gamma_{\resamp_n}}{\beta_{\resamp_n}+\gamma_{\resamp_n}}
= \prod_{n \ge 1} \frac{1}{1+\beta_{\resamp_n}/\gamma_{\resamp_n}}
,\]
given their respective masses. Our claim then reduces to showing that this product is close to $1$ with high probability when $\varepsilon$ is arbitrarily small.

Recall that the bad mass $\beta_{\resamp_n}$ is distributed as
\[(1-m) + h^2 \tau^2_{2^n-1}
\le \varepsilon (1 + \tau^2_{2^n-1}),\]
while the good mass is
\[\gamma_{\resamp_n} \ge (1-\varepsilon) \prod_{k=1}^n \frac{1+\resamp^1_k}{1+\resamp^1_{k-1}}
= (1-\varepsilon) \e^{n \Theta^1_n},\]
where the last equality serves as the definition of $\Theta^1_n$.

It is important to note here that $\tau^2_{2^n-1}$ and $\Theta^1_n$ do not depend on $\varepsilon$.
Fix a small $\delta>0$.
On the one hand $1+\tau^2_{2^n-1}$ grows slower than $2^{n (2+\delta)}$ and the other hand, by Corollary~\ref{cor:ergo_masses_resampling}, we know that $\Theta^1_n \to 2$ almost surely, so $\e^{n \Theta^1_n}$ grows faster than $\e^{n (2-\delta)}$.
It follows that if $\delta$ is small enough, then the ratio $\beta_{\resamp_n}/\gamma_{\resamp_n}$ is upper bounded by $\varepsilon c^n Z$ with $c = 2^{(2+\delta)} \e^{-(2-\delta)} < 1$ and where $Z$ is an almost surely finite random variable independent of $\varepsilon$. This implies that indeed the product $\prod_{n \ge 1} \frac{1}{1+\beta_{\resamp_n}/\gamma_{\resamp_n}}$ is positive and tends to $1$ in probability as $\varepsilon \to 0$.
\end{proof}

\subsection{Proof of ergodicity}

We now have all the ingredients to prove the claimed unique ergodicity, see~\cite{HMS11} for general arguments of this kind.

\begin{proof}[Proof of Theorem~\ref{thm:ergo}]
Let $\boldsymbol{\mathcal{T}}^{\, \bullet}_0$ be an arbitrary marked tree with mass $1$ and a marked point different from the root and let $\varepsilon > 0$ be arbitrarily small. According to Lemma~\ref{lem:ergo_good}, almost surely there exists a finite stopping time $t_0$, of the form $\resamp_N$, such that the rescaled tree $\boldsymbol{\mathcal{S}}^{\, \bullet}_{t_0}$ is $\varepsilon$-good.
Applying Proposition~\ref{prop:toy_ergo} to this tree $\boldsymbol{\mathcal{S}}^{\, \bullet}_{t_0}$ shows that with probability bounded below by $c_\varepsilon>0$, all subsequent resampled points will belong to the good Brownian subtree which will then (by Corollary~\ref{cor:ergo_masses_resampling}) grow more rapidly than the rest. Hence the GHP distance between the rescaled entire tree $(1+t_0)^{-1/2} \cdot \boldsymbol{\mathcal{T}}_{t_0}^{\, \bullet}$ and the good part $\gamma_{t_0}^{-1/2} \cdot \boldsymbol{\mathcal{G}}_{t_0}^{\, \bullet}$ tends to $0$.

According to Theorem~\ref{thm:cv_BBJ_entre_deux_resampling} and Lemma~\ref{lem:indep_tree_failure}, since the good part at time $t_0$ is a marked Brownian CRT conditionally on its size, then this remains true at each subsequent time.
This proves the convergence of the process $\boldsymbol{\mathcal{S}}^{\, \bullet}_t = (1+t)^{-1/2} \cdot \boldsymbol{\mathcal{T}}_{t}^{\, \bullet}$ to the standard Brownian CRT on this event with probability $c_\varepsilon$, and then in general since $c_\varepsilon$ can be made arbitrarily close to $1$ by taking $\varepsilon$ small enough by Corollary~\ref{cor:toy_ergo_gde_proba}.
\end{proof}

\phantomsection
\end{document}